\documentclass[a4paper,11pt,twoside]{article}
\usepackage{pst-fill,pst-grad}
\usepackage{textcomp}
\usepackage[english]{babel}
\usepackage{graphicx}
\usepackage{amsmath}
\usepackage{float}
\usepackage[matrix,arrow,curve]{xy}
\usepackage{pstricks}
\usepackage{mathtools}
\usepackage{amsmath,amsfonts,verbatim,afterpage,theorem,euscript,mathrsfs,amssymb}
\usepackage{amsfonts}
\usepackage{amssymb}
\usepackage{array}
\usepackage{dsfont}
\usepackage{hyperref}
\def \vu{\vec{u}}

\def \vf{\vec{f}}
\def \div{\mathrm{div}}
\def \vn{\vec{\nabla}}
\def \div{\mathrm{div}}

\def \vvU{\vec{\mathcal{U}}}
\def \vvF{\vec{\mathcal{F}}}

\DeclareFontFamily{U}{mathx}{}
\DeclareFontShape{U}{mathx}{m}{n}{<-> mathx10}{}
\DeclareSymbolFont{mathx}{U}{mathx}{m}{n}
\DeclareMathAccent{\widehat}{0}{mathx}{"70}
\DeclareMathAccent{\widecheck}{0}{mathx}{"71}

\newtheorem{Theoreme}{Theorem}

\newtheorem{Definition}{Definition}[section]
\newtheorem{Proposition}{Proposition}[section]
\newtheorem{Lemme}{Lemma}[section]

\newtheorem{Remarque}{Remark}[section]

\numberwithin{equation}{section}
\title{\bf Global Mild solutions for the fractional Navier-Stokes-Boussinesq equations in critical Fourier-Herz-Lorentz spaces}
\author{Diego Chamorro\footnote{Laboratoire de Math\'ematiques et Mod\'elisation d'Evry (LaMME) - UMR 8071. Universit\'e d'Evry Val d'Essonne, 23 Boulevard de France, 91037 Evry Cedex, France. email: \textit{diego.chamorro@univ-evry.fr}},  Maxence Mansais\footnote{Laboratoire de Math\'ematiques et Mod\'elisation d'Evry (LaMME) - UMR 8071. Universit\'e d'Evry Val d'Essonne, 23 Boulevard de France, 91037 Evry Cedex, France. email: \emph{maxence.mansais@ens-paris-saclay.fr}}.}

\begin{document}
%%%%%%%%%%%%%%%%%%%%%%%%%%%%%%%%%%%%%%%%%%%%%%%%%%%
\maketitle
%%%%%%%%%%%%%%%%%%%%%%%%%%%%%%%%%%%%%%%%%%%%%%%%%%%
\begin{abstract}
\noindent We construct here global in time mild solutions to the forced fractional Navier-Stokes-Boussinesq system in some critical Fourier-Herz spaces which will be based on Lorentz norms in the time and the frequency variables. For the initial data and for the external force we will consider Fourier-Besov spaces that will also be based on Lorentz spaces. Moreover, we will show, by performing a separate study of the velocity field and the temperature, that it is possible to consider one of the largest functional space -in the Fourier-Besov-Lorentz framework- for the initial velocity field. 
\end{abstract}
%%%%%%%%%%%%%%%%%%%%%%%%%%%%%%%%%%%%%%%%%%%%%%%%%%%
%\tableofcontents
%%%%%%%%%%%%%%%%%%%%%%%%%%%%%%%%%%%%%%%%%%%%%%%%%%%%%%%%%%%%%%%%%%%%%%%%%%%%%%%%%%%%%%%%%%%%%%%%%%%%%%
\section{Introduction}
%%%%%%%%%%%%%%%%%%%%%%%%%%%%%%%%%%%%%%%%%%%%%%%%%%%
We study in this article the following fractional incompressible, Navier-Stokes-Boussinesq system for a fixed dimension $d\geq2$:
\begin{equation}\label{eq_NSBalpha_Intro}
\begin{cases}
\partial_t\vec{u} = -(-\Delta)^\frac{\alpha}{2}\vec{u}-\mathrm{div}(\vec{u}\otimes \vec{u}) -\vec{\nabla}p + (\theta \vec{e_d})+\vec{f}, \quad t>0,x\in \mathbb{R}^d, \quad \mathrm{div}(\vec{u})=0,\\[4mm]
\partial_t\theta= -(-\Delta)^\frac{\alpha}{2}\theta-\mathrm{div}(\theta\vec{u}), \quad t>0, x\in \mathbb{R}^d\\[4mm]
\vec{u}(0,x)=\vu_0(x), \mathrm{div}(\vu_0)=0, \qquad \theta(0,x)=\theta_0(x), \quad x\in \mathbb{R}^d.
\end{cases}
\end{equation}
For the fractional powers of the Laplacian, which can be defined by the expression in the Fourier level $\widehat{\;(-\Delta)^{\frac{\alpha}{2}}\phi\;}(\xi)=|\xi|^\alpha\widehat{\phi}(\xi)$, we will consider the case $1<\alpha< 2$.\\

This system of coupled equations models the behaviour of geophysical fluids, indeed, the first line corresponds to the incompressible fractional Navier-Stokes system, where $\vu:[0,+\infty[\times \mathbb{R}^d\longrightarrow \mathbb{R}^d$ denotes the velocity of the fluid, $p:[0,+\infty[\times\mathbb{R}^d\longrightarrow \mathbb{R}$ is the internal pressure of the fluid, and $\vu_0:\mathbb{R}^d\longrightarrow \mathbb{R}^d$ its given divergence-free initial velocity. Here the function $\theta:[0,+\infty[\times \mathbb{R}^d\longrightarrow \mathbb{R}$ is the temperature and $\vf:[0,+\infty[\times \mathbb{R}^d\longrightarrow \mathbb{R}^d$ is a given external force. The second line gives the evolution of the temperature $\theta$ from a given initial temperature $\theta_0$ and its evolution is governed by a fractional diffusion and by the drift $\vu$. For more details on the modeling provided by the equations (\ref{eq_NSBalpha_Intro}), please refer to the books \cite{Pedlosky} or \cite{Salmon}.\\

In the equations above, the unknowns are the velocity field $\vu$, the pressure $p$ and the temperature $\theta$, however, by formally applying the divergence to the first equation of (\ref{eq_NSBalpha_Intro}) we obtain the equation
$$\Delta p=-\div \big(\div (\vu\otimes \vu) +\theta \vec{e}_d+\vf\big),$$
which will allow us to concentrate our study on the variables $\vu$ and $\theta$. In order to do so, we will use the Leray projector given by the expression $\mathbb{P}(\vec{\phi})=\vec{\phi}+\nabla (-\Delta)^{-1}\mathrm{div}(\vec{\phi})$, note in particular that we have $\mathbb{P}(\vn p)=0$ and $\mathbb{P}(\vu)=\vu$ since $\div(\vu)=0$, thus, if we apply the Leray projector to the previous system we obtain
\begin{equation}\label{eq_NSBalpha}
\begin{cases}
\partial_t\vec{u} = -(-\Delta)^\frac{\alpha}{2}\vec{u}-\mathbb{P}\mathrm{div}(\vec{u}\otimes \vec{u}) + \mathbb{P}(\theta \vec{e_d})+\mathbb{P}(\vec{f}), \quad t>0,x\in \mathbb{R}^d,\\[4mm]
\partial_t\theta= -(-\Delta)^\frac{\alpha}{2}\theta-\mathrm{div}(\theta\vec{u}), \quad t>0, x\in \mathbb{R}^d,\\[4mm]
\vec{u}(0,x)=\vu_0(x), \mathrm{div}(\vu_0)=0, \qquad \theta(0,x)=\theta_0(x).
\end{cases}
\end{equation}
This system of equations possesses different classes of symmetries, one of the main ones being the scaling symmetry. Indeed, observe that if $(\vec{u},\theta)$ solves (\ref{eq_NSBalpha}) for some initial data $(\vu_0,\theta_0,\vec{f})$, then for all $\lambda>0$, the rescaled variables $\vec{u}_{\lambda}(t,x)=\lambda^{\alpha-1}\vec{u}(\lambda^\alpha t,\lambda x)$ and $\theta_\lambda(t,x)=\lambda^{2\alpha-1}\theta(\lambda^\alpha t,\lambda x)$ solve the same system for the initial data $(\vu_{0,\lambda},{\theta_0}_\lambda,\vec{f}_\lambda)$ given by the rescaled functions
$$\vu_{0,\lambda}(t,x)=\lambda^{\alpha-1}\vu_0(\lambda x), \quad {\theta_0}_\lambda(t,x)=\lambda^{2\alpha-1}\theta_0(\lambda x), \quad \vec{f}_\lambda(t,x)=\lambda^{2\alpha-1}\vec{f}(\lambda^\alpha t,\lambda x).$$
This yields a natural definition of critical functional spaces, which will be useful to construct global in time solutions to the system (\ref{eq_NSBalpha}). Indeed, using Duhamel's formula, these equations can be rewritten in the following form
\begin{equation}\label{eq_Ptfix}
\begin{cases}
\displaystyle{\vec{u} = \mathfrak{p}_t* \vu_0-\int_0^t \mathfrak{p}_{t-s}*\mathbb{P}\left(\mathrm{div}(\vec{u}\otimes \vec{u})\right)ds+\int_0^t \mathfrak{p}_{t-s}*\mathbb{P}(\theta \vec{e}_d)ds + \int_0^t \mathfrak{p}_{t-s}*\mathbb{P}(\vec{f})ds},\\[4mm]
\displaystyle{\theta =\mathfrak{p}_t* \theta_0-\int_0^t \mathfrak{p}_{t-s}* \mathrm{div}(\theta \vec{u})ds},
\end{cases}
\end{equation}
and in order to solve (\ref{eq_Ptfix}) via a fixed point argument, we will consider a \emph{critical} functional framework given by Banach spaces such that we have the following invariances with respect to the dilations given above: $\|\vec{u}(\cdot, \cdot)\|_{E_{\vu}}=\|\lambda^{\alpha-1}\vec{u}(\lambda^\alpha \cdot,\lambda \cdot)\|_{E_{\vu}}$ for the velocity field, $\|\theta(\cdot, \cdot)\|_{E_{\theta}}=\|\lambda^{2\alpha-1}\theta(\lambda^\alpha \cdot,\lambda \cdot)\|_{E_{\theta}}$ for the temperature, $\|\vec{f}(\cdot, \cdot)\|_{E_{\vf}}=\|\lambda^{2\alpha-1}\vec{f}(\lambda^\alpha \cdot,\lambda \cdot)\|_{E_{\vf}}$ for the external force, $\|\vu_0(\cdot)\|_{\mathcal{E}_{\vu_0}}=\|\lambda^{\alpha-1}\vu_0(\lambda \cdot)\|_{\mathcal{E}_{\vu_0}}$ for the initial velocity and $\|\theta_0(\cdot)\|_{\mathcal{E}_{\theta_0}}=\|\lambda^{2\alpha-1}\theta_0(\lambda \cdot)\|_{\mathcal{E}_{\theta_0}}$ for the initial temperature. Thus, with such functional spaces we can apply the classical Banach contraction principle to obtain global in time solutions by proving suitable estimates on the initial data $\vu_0, \theta_0, \vf$, the linear term given by $\mathbb{P}(\theta \vec{e_d})$, as well as the quadratic terms defined by $\mathbb{P}(\mathrm{div}(\vec{u}\otimes \vec{u}))$ and $\mathrm{div}(\theta \vec{u})$.\\

Solutions to the problem (\ref{eq_NSBalpha_Intro}) (or some of its variants) and their properties were studied in \cite{BrandoleseSchonbek}, \cite{DanchinPaicu}, \cite{HmidiKR}, \cite{HmidiZerguine}, \cite{JiuMiao}, \cite{Sawada}, \cite{Stefanov}, \cite{YangJ} and \cite{YeXuXue}. In particular, global in time mild solutions for the problem (\ref{eq_Ptfix}) were obtained in our recent work \cite{Chamorro_Mansais1} where we used parabolic Morrey spaces.\\ 

In this article we are going to work in the Fourier level for the variables $\vec{u}$ and $\theta$ and this will lead us to a very different functional framework. Throughout this article, given a $\vec{\phi}$ function (either in the $x\in \mathbb{R}^d$ variable or in the $(t,x)\in [0,+\infty[\times \mathbb{R}^d$ variable) we will write by $\widehat{\vec{\phi}}$ the Fourier transform of $\vec{\phi}$ only taken in the $x$ variable and the frequency variable shall be always denoted as $\xi \in \mathbb{R}^d$. \\

We introduce now the following functional spaces that will be used here for the initial data: 
%%%%%%%%%%%%%%%%%%%%%%%%%%%%%%%%%%%%%%%%%%%%%%%%%%%
\begin{Definition}[Fourier-Besov-Lorentz spaces]\label{def_Fourier-Besov-Lorentz}
For $\mathfrak{s}\in \mathbb{R}$ and $1\leq p,q\leq +\infty$, we define the homogeneous Fourier-Besov-Lorentz space $\dot{\mathscr{B}}^{\mathfrak{s},p}_{L^{q,\infty}_\xi}(\mathbb{R}^d)$ as the space of all the tempered distributions $\vec{\psi}:\mathbb{R}^d\longrightarrow \mathbb{R}^d$ (modulo the polynomials, \emph{i.e.} $\vec{\psi}\in  \mathcal{S}'/ \mathcal{P}$) such that the following norm is finite
\begin{equation}\label{NormFBL}
\|\vec{\psi}\|_{\dot{\mathscr{B}}^{\mathfrak{s},p}_{L^{q,\infty}_\xi}}=\left(\displaystyle{\sum_{j\in \mathbb{Z}}\left(2^{j\mathfrak{s}}\left\|\mathds{1}_{\{2^j\leq |\xi| \leq 2^{j+1}\}}\widehat{\;\vec{\psi}\;}(\cdot)\right\|_{L^{q,\infty}}\right)^p}\right)^\frac{1}{p} \quad \textrm{if } 1\leq p<+\infty,
\end{equation}
with the usual modifications when $p=+\infty$ and where $L^{q,\infty}(\mathbb{R}^d)$ is the usual Lorentz space given by the quantity $\|\vec{\psi}\|_{L^{q,\infty}}=\underset{\lambda>0}{\sup}\left\{\lambda \left|\{x: |\vec{\psi}(x)|>\lambda\}\right|^{\frac{1}{q}}\right\}$ when $1<q<+\infty$, with the convention $L^{\infty, \infty}(\mathbb{R}^d)=L^\infty(\mathbb{R}^d)$ when $q=+\infty$.\\

\noindent We will also consider the space $\dot{\mathscr{B}}^{\mathfrak{s},p}_{L^{q}_\xi}(\mathbb{R}^d)$ for $1\leq q<+\infty$ where the Lorentz space $L^{q,\infty}_\xi(\mathbb{R}^d)$ is replaced by the usual Lebesgue space $L^q_\xi(\mathbb{R}^d)$ in the expression (\ref{NormFBL}) above.
\end{Definition}
\noindent We will also need the following time and space functional framework (the following spaces will be used as resolution spaces in our computations): 
%%%%%%%%%%%%%%%%%%%%%%%%%%%%%%%%%%%%%%%%%%%%%%%%%%%
\begin{Definition}[Time-Space Fourier-Herz-Lorentz spaces]\label{def_FourierLorentz}
Let $1\leq  p,q < +\infty$ be two real indexes and let $\mathfrak{s}\in \mathbb{R}$, we define the space $L^{p,\infty}_t([0,+\infty[,  \mathcal{F}^{\mathfrak{s}}_{L^{q,\infty}_\xi}(\mathbb{R}^d))$, also denoted shortly by $L^{p,\infty}_t(\mathcal{F}^{\mathfrak{s}}_{L^{q,\infty}_\xi})$, as the set of measurable functions $\vec{\phi}:[0,+\infty[\times\mathbb{R}^d\longrightarrow \mathbb{R}^d$ such that its weighted Fourier transform in the space variable $|\xi|^\mathfrak{s}\widehat{\;\vec{\phi}\;}(t, \xi)$ belongs to the Lorentz space $L^{q,\infty}(\mathbb{R}^d)$ and the resulting function in the time variable belongs to the Lorentz space $L^{p,\infty}_t([0,+\infty[)$. We equip this space with the norm
\begin{equation}\label{Def_NormFHL}
\|\vec{\phi}(\cdot,\cdot)\|_{L^{p,\infty}_t(\mathcal{F}^{\mathfrak{s}}_{L^{q,\infty}_\xi})} =\left\|\,\big\||\cdot|^{\mathfrak{s}}\widehat{\vec{\phi}}(\cdot,\cdot)\big\|_{L^{q,\infty}}\right\|_{L^{p,\infty}_t}.
\end{equation}
Note in particular that when $\mathfrak{s}=0$, we simply have 
$$\|\vec{\phi}(\cdot,\cdot)\|_{L^{p,\infty}_t(\mathcal{F}^{0}_{L^{q,\infty}_\xi})}=\|\,\|\widehat{\vec{\phi}}(\cdot,\cdot)\|_{L^{q,\infty}}\|_{L^{p,\infty}_t}=\|\vec{\phi}(\cdot,\cdot)\|_{L^{p,\infty}_t(L^{q,\infty}_\xi)}.$$
\end{Definition}
Remark that in the previous functional spaces defined with the expressions (\ref{NormFBL}) and (\ref{Def_NormFHL}), the Fourier transform is only considered in the space variable and not in the time variable. Let us mention here that this generic family of functional spaces for the initial data was considered before for the study of mild solutions in \cite{Aurazo}, \cite{CannoneWu}, \cite{Barakaa}, \cite{Ferreira},  \cite{LeiLin} and \cite{Xiao} (see also the references therein for more details about this functional framework). \\

\noindent With these functional spaces at hand, we can now state our first result.
%%%%%%%%%%%%%%%%%%%%%%%%%%%%%%%%%%%%%%%%%%%%%%%%%%%%%%%%%%%%%%%%%%%%%%%%%%%%%%%%%%%%%%%%%%%%%%%%%%%%%%
\begin{Theoreme}\label{thm_boussinesq_simple}
Let $1<\alpha < 2$ and $d\geq 2$. % $1<\alpha \leq 2$ if $d\geq 3$ and consider $1<\alpha<2$ if $d=2$. 
Let us fix the real parameters $1< p, q<+\infty$, $1<\mathfrak{p}, \mathfrak{q}<+\infty$, $1< {\bf p}, {\bf q}<+\infty$ and $\mathfrak{s}\in \mathbb{R}$ by the conditions:
\begin{itemize}
\item[1)]$\frac{\alpha}{p}+\frac{d}{q'}=\alpha-1$,  with $2<p<+\infty$, $1<q<2$,\\
\item[2)]$\frac{\alpha}{\mathfrak{p}}+\frac{d}{\mathfrak{q}'}=2\alpha-1$, with $1<p'<\mathfrak{p}<p$ and $q\leq \mathfrak{q}< q'$,\\
\item[3)]$\frac{\alpha}{{\bf p}}+\frac{d}{{\bf q}'}-\mathfrak{s}=2\alpha-1$ with  $1<{\bf p}<p$ and $q\leq {\bf q}<+\infty$.
\end{itemize}
Let $\vu_0$ be a divergence free initial data such that $\vu_0 \in \dot{\mathscr{B}}^{-\frac{\alpha}{p},\infty}_{L^{q,\infty}_\xi}(\mathbb{R}^d)$, let $\theta_0$ be an initial temperature that belongs to the space $\dot{\mathscr{B}}^{-\frac{\alpha}{\mathfrak{p}},\infty}_{L^{\mathfrak{q},\infty}_\xi}(\mathbb{R}^d)$ and consider an external force $\vec{f}$ that belongs to the space $L^{{\bf p},\infty}_t([0,+\infty[, \mathcal{F}^{\mathfrak{s}}_{L^{{\bf q},\infty}_\xi}(\mathbb{R}^d))$. If the quantity 
$$\|\vu_0\|_{\dot{\mathscr{B}}^{-\frac{\alpha}{p}, \infty}_{L^{q,\infty}_\xi}}+\|\theta_0\|_{\dot{\mathscr{B}}^{-\frac{\alpha}{\mathfrak{p}}, \infty}_{L^{\mathfrak{q},\infty}_\xi}}+\|\vec{f}\|_{L^{{\bf p},\infty}_t(\mathcal{F}^{\mathfrak{s}}_{L^{{\bf q},\infty}_\xi})},$$
is small enough, then there exists a couple $(\vec{u},\theta)$ which is a global mild solution to the fractional dissipative Navier-Stokes-Boussinesq system (\ref{eq_Ptfix}) and we have $(\vec{u},\theta)\in L^{p,\infty}_t(L^{q,\infty}_\xi)\times L^{\mathfrak{p},\infty}_t(L^{\mathfrak{q},\infty}_\xi)$.\\
\end{Theoreme}

%%%%%%%%%%%%%%%%%%%%%%%%%%%%%%%%%%%%%%%%%%%%%%%%%%%%%%%%%%%%%%%%%%%%%%%%%%%%%%%%%%%%%%%%%%%%%%%%%%%%%%
We give a few comments on this result. First note that the relationships $\frac{\alpha}{p}+\frac{d}{q'}=\alpha-1$, $\frac{\alpha}{\mathfrak{p}}+\frac{d}{\mathfrak{q}'}=2\alpha-1$ and $\frac{\alpha}{{\bf p}}+\frac{d}{{\bf q}'}-\mathfrak{s}=2\alpha-1$ correspond to the expected critical scaling for the velocity $\vec{u}$, the temperature $\theta$ and the external force $\vec{f}$, and this particular feature allows us to consider global in time mild solutions. Second, it is worth to recall here that the resolution spaces $L^{p,\infty}_t(L^{q,\infty}_\xi)$ and $L^{\mathfrak{p},\infty}_t(L^{\mathfrak{q},\infty}_\xi)$ do not involve any kind of regularity. Remark also that this result, due to the Fourier approach, is completely different from the previous results obtained in \cite{Chamorro_Mansais1} or \cite{YangJ}. Note finally, that in existing literature -and to the best of our knowledge- the system (\ref{eq_Ptfix}) was studied without external forces.\\

We note now that the study performed above is relatively symmetrical in the sense that the general structure of the functional spaces for $\vu$ and $\theta$ is essentially the same (and the same observation holds for the initial data $\vu_0$ and $\theta_0$): the only difference between these spaces is given by the relationships over the indexes. However, if we consider a slightly different framework in the study of the unknowns $\vu$ and $\theta$, we can improve the previous result in an interesting manner and in what follows we are going to consider a different functional framework for the velocity field $(\vu_0, \vu)$ and for the temperature $(\theta_0, \theta)$. Indeed, by analyzing the structure of the functional spaces $\dot{\mathscr{B}}^{s,p}_{L^{q,\infty}_\xi}(\mathbb{R}^d)$ it is possible to identify some inclusion properties between these spaces and in this sense we have the following result: 
%%%%%%%%%%%%%%%%%%%%%%%%%%%%%%%%%%%%%%%%%%%%%%%%%%%
\begin{Proposition}\label{prop_max_fourier_besov}
Consider $(E, \|\cdot\|_E)$ be a  Banach such that $\mathcal{S}(\mathbb{R}^d)\subset E\subset  \mathcal{S}'(\mathbb{R}^d)$. If we have the two following properties:
\begin{itemize}
\item[1)] there exists a constant $C>0$ such that  $\|\mathds{1}_{\{1\leq|\xi|\leq 2\}}\widehat{\vec{\psi}}(\cdot)\|_{L^1}\leq C\|\vec{\psi}\|_E$ for all $\vec{\psi}\in E$ and
\item[2)] we have the homogeneity property $\|\lambda^\beta \vec{\psi}(\lambda\cdot)\|_E=\|\vec{\psi}\|_E$ for all $\lambda>0$,
\end{itemize}
then we have the space inclusion $E\subset \dot{\mathscr{B}}^{-\beta, \infty}_{L^{1}_\xi}(\mathbb{R}^d)$. \\

\noindent In particular, if for some indexes $1<p,q<+\infty$ and $1<\alpha<2$ we have the relationship $\frac{\alpha}{p}+\frac{d}{q'}=\beta$ then we have the space inclusion
$$\dot{\mathscr{B}}^{-\frac{\alpha}{p}, \infty}_{L^{q,\infty}_\xi}(\mathbb{R}^d)\subset \dot{\mathscr{B}}^{-\beta, \infty}_{L^{1}_\xi}(\mathbb{R}^d).$$
\end{Proposition}
%%%%%%%%%%%%%%%%%%%%%%%%%%%%%%%%%%%%%%%%%%%%%%%%%%%
The proof of this proposition is postponed to the Appendix \ref{Appendix1}. Note that we trivially have the space inclusion $\dot{\mathscr{B}}^{-\beta, \infty}_{L^{1}_\xi}(\mathbb{R}^d)\subset \dot{\mathscr{B}}^{-\beta, \infty}_{L^{1,\infty}_\xi}(\mathbb{R}^d)$, since we always have the  embedding $L^1\subset L^{1,\infty}$ between the Lebesgue space $L^1$ and the Lorentz space $L^{1,\infty}$. However, the larger Lorentz space $L^{1,\infty}$ will not be suitable for our purposes as it does not have a convenient behavior when dealing with Young inequalities for convolution (see more details in Section \ref{Secc_LorentzProperties} and Remark \ref{Remarque1} below) and this particular fact forces us to consider the smaller Lebesgue space $L^1$ instead of the Lorentz space $L^{1,\infty}$. \\

Now we are going to use this previous proposition to perform, as announced, a separated study of the velocity field and of the temperature and this will lead us to a generalization of the Theorem \ref{thm_boussinesq_simple}.
%%%%%%%%%%%%%%%%%%%%%%%%%%%%%%%%%%%%%%%%%%%%%%%%%%%%%%%%%%%%%%%%%%%%%%%%%%%%%%%%%%%%%%%%%%%%%%%%%%%%%%
\begin{Theoreme}\label{thm_boussinesq_L1_alpha}
Let $d\geq 2$ and consider a diffusion parameter $1<\alpha < 2$ in the equation (\ref{eq_NSBalpha}).\\

\noindent Let $\vu_0:\mathbb{R}^d\longrightarrow \mathbb{R}^d$ be a divergence free initial velocity field such that $\vu_0 \in \dot{\mathscr{B}}^{-(\alpha-1), \infty}_{L^{1}_\xi}(\mathbb{R}^d)$. Let $\theta_0:\mathbb{R}^d\longrightarrow \mathbb{R}$ be an initial temperature such that $\theta_0\in \dot{\mathscr{B}}^{-\frac{\alpha}{\mathfrak{p}}, \infty}_{L^{\mathfrak{q},\infty}_\xi}(\mathbb{R}^d)$ where we have the relationship $\frac{\alpha}{\mathfrak{p}}+\frac{d}{\mathfrak{q}'}=2\alpha-1$ and the condition $1<\alpha<\mathfrak{p}<\frac{\alpha}{\alpha-1}$. Consider moreover an external force $\vf:[0,+\infty[\times \mathbb{R}^d\longrightarrow \mathbb{R}^d$ such that $L^{{\bf p},\infty}_t([0,+\infty[, \mathcal{F}^{\mathfrak{s}}_{L^{{\bf q},\infty}_\xi}(\mathbb{R}^d))$ under the condition $\frac{\alpha}{\bf p}+\frac{d}{{\bf q}'}-\mathfrak{s}=2\alpha-1$ with $1<{\bf p}<\frac{\alpha}{\alpha-1}$, $1< {\bf q}<+\infty$ and $0\leq {\mathfrak{s}}<\frac{d}{{\bf q}'}$.\\

\noindent If the quantity $\|\vu_0\|_{\dot{\mathscr{B}}^{-(\alpha-1), \infty}_{L^{1}_\xi}}+\|\theta_0\|_{ \dot{\mathscr{B}}^{-\frac{\alpha}{\mathfrak{p}}, \infty}_{L^{\mathfrak{q},\infty}_\xi}}+\|\vec{f}\|_{L^{{\bf p},\infty}_t(\mathcal{F}^{\mathfrak{s}}_{L^{{\bf q},\infty}_\xi})}$ is small enough then there exists a global mild solution $(\vec{u},\theta)$ to the fractional dissipative Navier-Stokes-Boussinesq system  (\ref{eq_NSBalpha}) such that we have $(\vec{u},\theta)\in L^{\frac{\alpha}{\alpha-1},\infty}_t(L^1_\xi)\times L^{\mathfrak{p},\infty}_t(L^{\mathfrak{q},\infty}_\xi)$.\\
\end{Theoreme}
%%%%%%%%%%%%%%%%%%%%%%%%%%%%%%%%%%%%%%%%%%%%%%%%%%%%%%%%%%%%%%%%%%%%%%%%%%%%%%%%%%%%%%%%%%%%%%%%%%%%%%
Let us remark that, from the point of view of the homogeneity, the functional spaces $\dot{\mathscr{B}}^{-(\alpha-1), \infty}_{L^{1}_\xi}$, $\dot{\mathscr{B}}^{-\frac{\alpha}{\mathfrak{p}}, \infty}_{L^{\mathfrak{q},\infty}_\xi}$ and $L^{{\bf p},\infty}_t(\mathcal{F}^{\mathfrak{s}}_{L^{{\bf q},\infty}_\xi})$ satisfy the expected stability identities with respect to the dilations as we have $\|\lambda^{\alpha-1}\vu_0(\lambda \cdot)\|_{\dot{\mathscr{B}}^{-(\alpha-1), \infty}_{L^{1}_\xi}} \simeq \|\vu_0(\cdot)\|_{\dot{\mathscr{B}}^{-(\alpha-1), \infty}_{L^{1}_\xi}}$, $\|\lambda^{2\alpha-1}\theta_0(\lambda \cdot)\|_{\dot{\mathscr{B}}^{-\frac{\alpha}{\mathfrak{p}}, \infty}_{L^{\mathfrak{q},\infty}_\xi}} \simeq \|\theta_0(\cdot)\|_{\dot{\mathscr{B}}^{-\frac{\alpha}{\mathfrak{p}}, \infty}_{L^{\mathfrak{q},\infty}_\xi}}$ and \\ $\|\lambda^{2\alpha-1}\vec{f}(\lambda^\alpha \cdot,\lambda \cdot)\|_{L^{{\bf p},\infty}_t(\mathcal{F}^{\mathfrak{s}}_{L^{{\bf q},\infty}_\xi})}=\|\vec{f}(\cdot, \cdot)\|_{L^{{\bf p},\infty}_t(\mathcal{F}^{\mathfrak{s}}_{L^{{\bf q},\infty}_\xi})}$. Note also that,  due to the Proposition \ref{prop_max_fourier_besov}, the functional space considered for the initial data $\vu_0$ is more general than the one considered in the previous Theorem \ref{thm_boussinesq_simple} but the framework for the initial temperature $\theta_0$ and the external force $\vf$ remains essentially the same. The only difference is based on the conditions over the parameters: in this result the lower and upper bounds for the indexes $\mathfrak{p}$ and ${\bf p}$ depend explicitly on $\alpha$. These two results are however different, specially if we focus our study on the external force. Indeed, note that the condition imposed for the index ${\bf p}$ in Theorem \ref{thm_boussinesq_L1_alpha} is $1<{\bf p}<\frac{\alpha}{\alpha-1}$ with $1<\alpha<2$, while the condition \emph{3)} of the Theorem \ref{thm_boussinesq_simple} is simply ${\bf p}<p$ where $2<p<+\infty$ and thus it is not hard to consider an external force $\vf$ that satisfies the hypotheses of one of these two results but that it is not suitable for the other one. \\

To the best of our knowledge these two results are new for the setting of the fractional incompressible 
Navier-Stokes-Boussinesq system (\ref{eq_NSBalpha_Intro}) and we will compare in the Appendix \ref{Appendix2} the result above with the ones obtained previously in \cite{Chamorro_Mansais1}. Although we do not claim any kind of optimality here, we believe that the use of the larger space $\dot{\mathscr{B}}^{-(\alpha-1), \infty}_{L^{1,\infty}_\xi}$ for the initial data $\vu_0$, as well as the space $L^{\frac{\alpha}{\alpha-1},\infty}_t(L^{1,\infty}_\xi)$ as a resolution space for the velocity field, seem not to be well suited to perform a fixed point argument. Thus, due to the Proposition \ref{prop_max_fourier_besov}, the Theorem \ref{thm_boussinesq_L1_alpha} contains the existing previous results where the initial data is studied in the framework of Fourier-based functional spaces. Nevertheless, and just as for the classical Navier-Stokes equations, the problem of finding the biggest functional framework (for initial data in the Fourier-based setting) where to construct a critical mild solution seems to be a very challenging open problem for the fractional Navier-Stokes-Boussinesq system studied here.\\

%%%%%%%%%%%%%%%%%%%%%%%%%%%%%%%%%%%%%%%%%%%%%%%%%%%%%%%%%%%%%%%%%%%%%%%%%%%%%%%%%%%%%%%%%%%%%%%%%%%%%%
The structure of the article is the following. In Section 2, we recall the definition 	and some useful estimates for Lorentz spaces, in Section 3 we prove the Theorem \ref{thm_boussinesq_simple} and in Section 4 we prove the Theorem \ref{thm_boussinesq_L1_alpha}. As announced, the Appendix \ref{Appendix1} is devoted to the proof of the Proposition \ref{prop_max_fourier_besov}, while in the Appendix \ref{Appendix2} we will study the relationship between the resolution space $L^{\frac{\alpha}{\alpha-1}, \infty}_t(L^1_\xi)$ and the parabolic Morrey spaces used in \cite{Chamorro_Mansais1}.
%%%%%%%%%%%%%%%%%%%%%%%%%%%%%%%%%%%%%%%%%%%%%%%%%%%%%%%%%%%%%%%%%%%%%%%%%%%%%%%%%%%%%%%%%%%%%%%%%%%%%%
\section{Properties of Lorentz spaces}\label{Secc_LorentzProperties}
%%%%%%%%%%%%%%%%%%%%%%%%%%%%%%%%%%%%%%%%%%%%%%%%%%%
We recall here the main characterizations as well as some properties of the Lorentz spaces which are the main tool of this work, for more details see the book \cite{Chamorro_Lorentz}. For a function $f:\mathbb{R}^d\longrightarrow \mathbb{R}$, we consider first its distribution function $d_f(\lambda)=\left|\{x\in \mathbb{R}^d: |f(x)|>\lambda\} \right|$, for any $\lambda>0$. Then, for a parameter $1\leq p<+\infty$, we define the Lorentz space $L^{p,\infty}(\mathbb{R}^d)$ by the condition
\begin{equation}\label{Def_Lorentz1}
\|f\|_{L^{p,\infty}}=\underset{\lambda >0}{\sup}\left\{\lambda\; d_{f}^{\frac1p}(\lambda)\right\}=\underset{\lambda >0}{\sup}\left\{\lambda\;|\{x\in \mathbb{R}^d: |f(x)|>\lambda\}|^{\frac1p}\right\}<+\infty.
\end{equation}
In the case when $p=+\infty$, we simply set $L^{\infty, \infty}(\mathbb{R}^d)=L^\infty(\mathbb{R}^d)$. We will need another characterization of the Lorentz spaces and we introduce now the decreasing rearrangement function of $f$ by the expression $f^*(t) = \underset{\alpha>0}{\inf} \{d_f(\alpha) \leq t\}$. 
Thus, for any $1\leq p<+\infty$ and $1\leq q<+\infty$, we define the Lorentz space $L^{p,q}(\mathbb{R}^d)$ by the condition 
$$\|f\|_{L^{p,q}}=\left(\int_{0}^{+\infty}\left(t^{\frac 1p }f^{\ast}(t)\right)^{q}\frac{dt}{t}\right)^{\frac 1q},$$
when $q=+\infty$, we have 
\begin{equation}\label{Lorentzinfty}
\|f\|_{L^{p,\infty}}=\underset{t>0}{\sup} \;\left\{ t^{\frac 1p}f^{\ast}(t)\right\},
\end{equation}
and these characterizations are equivalent. We recall that we have the following homogeneity property: for $\lambda>0$, if we set $f_\lambda(x)=f(\lambda x)$, we have 
\begin{equation}\label{HomogeneiteLorentz}
\|f_\lambda\|_{L^{p,q}}=\lambda^{-\frac{d}{p}}\|f\|_{L^{p,q}},
\end{equation}
for all $1\leq p, q\leq +\infty$. Note moreover that if $f(x)=|x|^{-\frac{d}{p}}$ then we have $f\in L^{p,\infty}(\mathbb{R}^d)$ (we will use this particular function later on) and this type of function is the simplest example that shows that we have $L^p(\mathbb{R}^d)\subsetneq L^{p,\infty}(\mathbb{R}^d)$.\\

Recall now that we have the H\"older inequality
\begin{equation}\label{Holder1}
\|fg\|_{L^{p,\infty}}\leq \|f\|_{L^{p_0,\infty}}\|g\|_{L^{p_1,\infty}},
\end{equation}
where $\frac{1}{p}=\frac{1}{p_0}+\frac{1}{p_1}$ with $1\leq p, p_0, p_1<+\infty$ (see \cite[Corolario 1.1.3]{Chamorro_Lorentz}).  We also have the following version of the H\"older inequalities
\begin{equation}\label{Holder2}
\|fg\|_{L^{1}}\leq C\|f\|_{L^{p_0,1}}\|g\|_{L^{p_1,\infty}},
\end{equation}
where $1=\frac{1}{p_0}+\frac{1}{p_1}$ with $1< p, p_0, p_1<+\infty$ (see \cite[Teorema 1.2.6]{Chamorro_Lorentz}).\\

For the Young inequality for convolution we have
\begin{equation}\label{Young1}
\|f\ast g\|_{L^{p,\infty}}\leq \|f\|_{L^{p_0,\infty}}\|g\|_{L^{p_1,\infty}},
\end{equation}
with $1+\frac{1}{p}=\frac{1}{p_0}+\frac{1}{p_1}$ and where $1<p, p_0, p_1<+\infty$ (see \cite[Teorema 1.4.8]{Chamorro_Lorentz}).  A particular case of the Young inequalities is the following (see \cite[Teorema 1.1.6]{Chamorro_Lorentz}):
\begin{equation}\label{Young2}
\|f*g\|_{L^{p,\infty}}\leq C\|f\|_{L^{1}}\|g\|_{L^{p,\infty}}.
\end{equation}
%%%%%%%%%%%%%%%%%%%%%%%%%%%%%%%%%%%%%%%%%%%%%%%%%%%
\begin{Remarque}\label{Remarque1}
Note however that we do not have the estimate: 
$$\|f*g\|_{L^{p,\infty}}\leq C\|f\|_{L^{1,\infty}}\|g\|_{L^{p,\infty}}.$$
\end{Remarque}
Indeed, over $\mathbb{R}$, if we consider $f(x)=\frac{1}{|x|}$ and  $g(x)=\mathds{1}_{[0,1]}(x)$, we have $\|f\|_{L^{1,\infty}}<+\infty$ as well as $\|g\|_{L^{p,\infty}}<+\infty$ for all $1\leq p\leq+\infty$ but $f\ast g$ is not well defined as we have $f\ast g\equiv +\infty$ over the interval $[0,1]$.\\ 

This particular behavior of the Lorentz space $L^{1,\infty}$ with respect of the convolution makes this space unsuitable (to the best of our knowledge) for our calculations.

%%%%%%%%%%%%%%%%%%%%%%%%%%%%%%%%%%%%%%%%%%%%%%%%%%%%%%%%%%%%%%%%%%%%%%%%%%%%%%%%%%%%%%%%%%%%%%%%%%%%%%
\section{Proof of the Theorem \ref{thm_boussinesq_simple}}\label{Secc_ProofTh1}
%%%%%%%%%%%%%%%%%%%%%%%%%%%%%%%%%%%%%%%%%%%%%%%%%%%%%%%%%%%%%%%%%%%%%%%%%%%%%%%%%%%%%%%%%%%%%%%%%%%%%%
First note that the integral formulation (\ref{eq_Ptfix}) can be rewritten in the following manner
\begin{equation}\label{Equation_PointFixe}
\vvU=\vvU_0+\vvF+L(\vvU)+B(\vvU,\vvU),
\end{equation}
where $\vvU=\left[\begin{array}{c}\vu\\ \theta\end{array}\right]$ is a $d+1$ vector, $\vvU_0 = \left[\begin{array}{c}
\mathfrak{p}_t\ast \vec{u_0} \\\mathfrak{p}_t\ast \theta_0\end{array}\right]$ is the initial data, $\vvF = \left[\begin{array}{c} \displaystyle{\int_0^t} \mathfrak{p}_{t-s}\ast\mathbb{P}(\vec{f})ds\\
0\end{array}\right]$ is the given external force, the quantity $L(\vvU)=\left[\begin{array}{c} \displaystyle{\int_0^t} \mathfrak{p}_{t-s}\ast\mathbb{P}(\theta \vec{e}_d)ds\\  0\end{array}\right]$ is a linear term and finally $B(\vvU,\vvU)=\left[\begin{array}{c} -\displaystyle{\int_0^t} \mathfrak{p}_{t-s}\ast\mathbb{P}\left(\mathrm{div}(\vec{u}\otimes \vec{u})\right)ds\\     -\displaystyle{\int_0^t} \mathfrak{p}_{t-s}\ast \mathrm{div}(\theta \vec{u})ds\end{array}\right]$ is a bilinear term.\\ 
 
\noindent The equations of the form (\ref{Equation_PointFixe}) can be easily studied as we have the following version of the Banach fixed point theorem: 
%%%%%%%%%%%%%%%%%%%%%%%%%%%%%%%%%%%%%%%%%%%%%%%%%%%
\begin{Lemme}\label{Lemme_Banach_Picard}
Let $(E,\|\cdot\|_E)$ be a Banach space. Consider $\vvU_0 \in E$, let $L:E \longrightarrow E$ be a linear operator and $B:E\times E \longrightarrow E$ be a bilinear operator that satisfy the estimates
\begin{eqnarray*}
\|L(\vvU)\|_E &\leq &C_L \|\vvU\|_E,\\[2mm]
\|B(\vvU,\vec{\mathcal{V}})\|_E &\leq &C_B \|\vvU\|_E \|\vec{\mathcal{V}}\|_E, 
\end{eqnarray*}
 for all $\vvU, \vec{\mathcal{V}}\in E$. Assume moreover that $C_L<\frac{1}{3}$. If  we have $\|\vvU_0\|_E \leq \delta/2$ and $\|\vec{\mathcal{F}}\|_E \leq \delta/2$ for a constant $\delta>0$ such that $\delta < \frac{1}{9C_B}$, then there exists a unique $\vvU \in E$ such that $\|\vvU\|_E\leq 3 \delta$ which is a solution of the equation (\ref{Equation_PointFixe}).
\end{Lemme}
%%%%%%%%%%%%%%%%%%%%%%%%%%%%%%%%%%%%%%%%%%%%%%%%%%%
See \cite[Théorème 7.1.1]{Chamorro_Livre} for a proof of this lemma.\\

\noindent For the vector $\vvU=\left[\begin{array}{c}\vu\\ \theta\end{array}\right]$, as announced in the Theorem \ref{thm_boussinesq_simple}, we will consider as resolution space the functional space $E=L^{p,\infty}_t(L^{q,\infty}_\xi)\times L^{\mathfrak{p},\infty}_t(L^{\mathfrak{q},\infty}_\xi)$ endowed with the norm 
\begin{equation}\label{Def_NormeGlobale}
\|\vvU\|_{E}=\left\|\left[\begin{array}{c}\vu\\ \theta\end{array}\right]\right\|_{E}=\|\vu\|_{L^{p,\infty}_t(L^{q,\infty}_\xi)}+\mathfrak{C}\|\theta\|_{ L^{\mathfrak{p},\infty}_t(L^{\mathfrak{q},\infty}_\xi)},
\end{equation}
where $\mathfrak{C}=\mathfrak{C}(\alpha,d)\gg 1$ is a constant that will be fixed later and whose sole role is to absorb properly some of the constants in the previous Lemma \ref{Lemme_Banach_Picard} (see also our recent work \cite{Chamorro_Mansais1}). In this context, we thus need to prove the following estimates:
\begin{itemize}
\item[(1)] For the initial data $\vvU_0$:
\begin{eqnarray}
\|\vvU_0\|_{E}= \left\|\left[\begin{array}{c}
\mathfrak{p}_t\ast \vec{u_0} \\\mathfrak{p}_t\ast \theta_0\end{array}\right]\right\|_{E}&=&\|\mathfrak{p}_t\ast\vu_0\|_{L^{p,\infty}_t(L^{q,\infty}_\xi)}+\mathfrak{C}\|\mathfrak{p}_t\ast\theta_0\|_{ L^{\mathfrak{p},\infty}_t(L^{\mathfrak{q},\infty}_\xi)}\notag\\
&\leq & C\|\vu_0\|_{\dot{\mathscr{B}}^{-\frac{\alpha}{p},\infty}_{L^{q,\infty}_\xi}}+C\|\theta_0\|_{\dot{\mathscr{B}}^{-\frac{\alpha}{\mathfrak{p}},\infty}_{L^{\mathfrak{q},\infty}_\xi}}.\label{Estimation_DonneesInitiales}
\end{eqnarray}
\item[(2)] For the external force $\vvF$:
\begin{equation}
\|\vvF\|_{E} =\left\|\left[\begin{array}{c} \displaystyle{\int_0^t} \mathfrak{p}_{t-s}\ast\mathbb{P}(\vec{f})ds\\
0\end{array}\right]\right\|_{E}=\left\|\int_0^t\mathfrak{p}_{t-s}\ast\mathbb{P}(\vec{f})ds\right\|_{L^{p,\infty}_t(L^{q,\infty}_\xi)}\leq C\|\vec{f}\|_{L^{{\bf p},\infty}_t(\mathcal{F}^{\mathfrak{s}}_{L^{{\bf q},\infty}_\xi})}.\label{Estimation_Force}
\end{equation}
\item[(3)] For the Linear term $L(\vvU)$
\begin{eqnarray}
\|L(\vvU)\|_{E}=\left\|\left[\begin{array}{c} \displaystyle{\int_0^t}\mathfrak{p}_{t-s}\ast\mathbb{P}(\theta \vec{e}_d)ds\\  0\end{array}\right]\right\|_{E}&=&\left\|\int_0^t\mathfrak{p}_{t-s}\ast\mathbb{P}(\theta \vec{e}_d)ds\right\|_{L^{p,\infty}_t(L^{q,\infty}_\xi)}\notag \\
&\leq & C \|\theta\|_{L^{\mathfrak{p},\infty}_t(L^{\mathfrak{q},\infty}_\xi)}.\label{Estimation_Lineaire}
\end{eqnarray}
\item[(4)] For the Bilinear term $B(\vvU,\vvU)$
\begin{eqnarray}
\|B(\vvU,\vvU)\|_{E}&=&\left\|\left[\begin{array}{c} -\displaystyle{\int_0^t} \mathfrak{p}_{t-s}\ast\mathbb{P}\left(\mathrm{div}(\vec{u}\otimes \vec{u})\right)ds\\     -\displaystyle{\int_0^t} \mathfrak{p}_{t-s}\ast \mathrm{div}(\theta \vec{u})ds\end{array}\right]\right\|_{E}\notag\\
&=&\left\|\int_0^t\mathfrak{p}_{t-s}\ast\mathbb{P}\left(\mathrm{div}(\vec{u}\otimes \vec{u})\right)ds\right\|_{L^{p,\infty}_t(L^{q,\infty}_\xi)}+\mathfrak{C}\left\|\int_0^t \mathfrak{p}_{t-s}\ast \mathrm{div}(\theta \vec{u})ds\right\|_{L^{\mathfrak{p},\infty}_t(L^{\mathfrak{q},\infty}_\xi)}\notag\\
&\leq & C\|\vu\|_{L^{p,\infty}_t(L^{q,\infty}_\xi)}\|\vu\|_{L^{p,\infty}_t(L^{q,\infty}_\xi)}+C\|\vu\|_{L^{p,\infty}_t(L^{q,\infty}_\xi)}\|\theta\|_{L^{\mathfrak{p},\infty}_t(L^{\mathfrak{q},\infty}_\xi)}.\label{Estimation_Bilineaire}
\end{eqnarray}
\end{itemize}

The estimates (\ref{Estimation_DonneesInitiales}), (\ref{Estimation_Force}), (\ref{Estimation_Lineaire}) and (\ref{Estimation_Bilineaire}) will be treated separately.\\

\begin{itemize}
\item[(1)] We start with the initial data $\vu_0$ and $\theta_0$ and we will show that we have the controls 
$$\|\mathfrak{p}_t\ast\vu_0\|_{L^{p,\infty}_t(L^{q,\infty}_\xi)} \leq C \|\vu_0\|_{\dot{\mathscr{B}}^{-\frac{\alpha}{p},\infty}_{L^{q,\infty}_\xi}}\quad \mbox{and}\quad \|\mathfrak{p}_t\ast\theta_0\|_{ L^{\mathfrak{p},\infty}_t(L^{\mathfrak{q},\infty}_\xi)}\leq C\|\theta_0\|_{\dot{\mathscr{B}}^{-\frac{\alpha}{\mathfrak{p}},\infty}_{L^{\mathfrak{q},\infty}_\xi}},$$
but since these two terms have the same structure, for the sake of conciseness we will only prove the first estimate. \\

Let us first fix $t>0$ and $\lambda>0$. Since the sets $\{2^j \leq |\xi| < 2^{j+1}\}$ are mutually disjoint and cover $\mathbb{R}^d\setminus \{0\}$, we can write
$$    |\{\xi \ : \ e^{-t|\xi|^\alpha} |\widehat{\vu}_0(\xi)|>\lambda\}| =  \displaystyle{\sum_{j\in \mathbb{Z}}} |\{2^j\leq |\xi| < 2^{j+1} \ : \ |\widehat{\vu}_0(\xi)|>\lambda e^{t|\xi|^\alpha}\}|.$$
We now fix $j\in \mathbb{Z}$. By noticing that for $|\xi|\geq 2^j$ we have $e^{t|\xi|^\alpha}\geq e^{t2^{j\alpha}}$, we obtain the upper bound
$$|\{\xi \ : \ e^{-t|\xi|^\alpha} |\widehat{\vu}_0(\xi)|>\lambda\}| \leq \displaystyle{\sum_{j\in \mathbb{Z}}} |\{\xi \ : \ \mathds{1}_{\{2^j\leq |\xi| < 2^{j+1}\}}|\widehat{\vu}_0(\xi)|>\lambda e^{t2^{j\alpha}}\}|,$$
by taking the $L^{q,\infty}$ norm of the functions $\mathds{1}_{\{2^j\leq |\xi| < 2^{j+1}\}}\widehat{\vu}_0$, we obtain
$$|\{\xi \ : \ e^{-t|\xi|^\alpha |}\widehat{\vu}_0(\xi)|>\lambda\}| \leq  \sum_{j\in \mathbb{Z}} \left\|\mathds{1}_{\{2^j\leq |\xi| < 2^{j+1}\}}\widehat{\vu}_0(\xi)\right\|_{L^{q,\infty}}^q\times (\lambda e^{t2^{j\alpha}})^{-q}.$$
Since we have for the initial data $\vu_0\in \dot{\mathscr{B}}^{-\frac{\alpha}{p},\infty}_{L^{q,\infty}_\xi}(\mathbb{R}^d)$ by hypothesis and recalling the expression of the norm $\|\vu_0\|_{\dot{\mathscr{B}}^{-\frac{\alpha}{p},\infty}_{L^{q,\infty}_\xi}}=\underset{j\in \mathbb{Z}}{\sup} 2^{-j\frac{\alpha}{p}}\left\|\mathds{1}_{\{2^j\leq |\xi| < 2^{j+1}\}}\widehat{\vu}_0(\xi)\right\|_{L^{q,\infty}}$ given in the formula (\ref{NormFBL}), we can deduce the following control, valid for all $j \in \mathbb{Z}$
$$ \|\mathds{1}_{\{2^j\leq |\xi| < 2^{j+1}\}}\widehat{\vu}_0(\xi)\|_{L^{q,\infty}}  \leq  \|\vu_0\|_{ \dot{\mathscr{B}}^{-\frac{\alpha}{p},\infty}_{L^{q,\infty}_\xi}} 2^{j\frac{\alpha}{p}}.$$
We insert this estimate in the previous inequality to get
\begin{eqnarray*}
 |\{\xi \ : \ e^{-t|\xi|^\alpha} |\widehat{\vu}_0(\xi)|>\lambda\}| & \leq & \displaystyle{\sum_{j\in \mathbb{Z}}} \left(\|\vu_0\|_{ \dot{\mathscr{B}}^{-\frac{\alpha}{p},\infty}_{L^{q,\infty}_\xi}} 2^{j\frac{\alpha}{p}}\right)^q\times (\lambda e^{t2^{j\alpha}})^{-q}\\
& \leq & \lambda^{-q}\|\vu_0\|_{\dot{\mathscr{B}}^{-\frac{\alpha}{p},\infty}_{L^{q,\infty}_\xi}}^q \times \displaystyle{\sum_{j\in \mathbb{Z}}} 2^{j\frac{\alpha q}{p}}\times  e^{-qt2^{j\alpha}}\\
&\leq &\lambda^{-q} t^{-\frac{q}{p}}\|\vu_0\|_{ \dot{\mathscr{B}}^{-\frac{\alpha}{p},\infty}_{L^{q,\infty}_\xi}}^q \times \sum_{j\in \mathbb{Z}} (t2^{j\alpha})^{\frac{q}{p}}\times  e^{-qt2^{j\alpha}}.
\end{eqnarray*}
At this point, we take in the previous expression the $\frac{1}{q}-$th power and the supremum with respect to $\lambda>0$ in order to reconstruct the Lorentz norm $L^{q,\infty}$ (recall the expression (\ref{Def_Lorentz1}) above) and we have
\begin{eqnarray*}
\lambda  |\{\xi \ : \ e^{-t|\xi|^\alpha} |\widehat{\vu}_0(\xi)|>\lambda\}|^{\frac{1}{q}}&\leq &  t^{-\frac{1}{p}}\|\vu_0\|_{ \dot{\mathscr{B}}^{-\frac{\alpha}{p},\infty}_{L^{q,\infty}_\xi}}\left( \displaystyle{\sum_{j\in \mathbb{Z}}} (t2^{j\alpha})^{\frac{q}{p}}\times  e^{-qt2^{j\alpha}}\right)^\frac{1}{q}\\
\|e^{-t|\xi|^\alpha} |\widehat{\vu}_0|\|_{L^{q,\infty}} & \leq &  t^{-\frac{1}{p}}\|\vu_0\|_{ \dot{\mathscr{B}}^{-\frac{\alpha}{p},\infty}_{L^{q,\infty}_\xi}}\left( \displaystyle{\sum_{j\in \mathbb{Z}}} (t2^{j\alpha})^{\frac{q}{p}}\times  e^{-qt2^{j\alpha}}\right)^\frac{1}{q}.
\end{eqnarray*}
We will now prove that the quantity $\displaystyle{\sum_{j\in \mathbb{Z}}} (t2^{j\alpha})^{\frac{q}{p}}\times  e^{-qt2^{j\alpha}}$ can be bounded uniformly independently from $t>0$. Indeed, for $t>0$, there exists a unique $k\in \mathbb{Z}$ such that $2^{-k\alpha} \leq t <2^{-k\alpha+\alpha}$, we thus have $(t2^{j\alpha})^\frac{q}{p}e^{-qt2^{j\alpha}} \leq 2^{(-k\alpha+\alpha+j\alpha)\frac{q}{p}}e^{-q2^{-k\alpha+j\alpha}}=2^{\frac{\alpha q}{p}} 2^{\frac{\alpha q}{p}(j-k)} e^{-q2^{\alpha(j-k)}}$ and we use this upper bound to obtain
$$\sum_{j\in \mathbb{Z}} (t2^{j\alpha})^{\frac{q}{p}}\times  e^{-qt2^{j\alpha}} \leq 2^{\frac{\alpha q}{p}}  \sum_{j\in \mathbb{Z}} 2^{\frac{\alpha q}{p}(j-k)} e^{-q2^{\alpha(j-k)}},$$
and by a change of variables in the sum (from $j$ to $j-k$), we obtain 
$$\sum_{j\in \mathbb{Z}} (t2^{j\alpha})^{\frac{q}{p}}\times  e^{-qt2^{j\alpha}} \leq  2^{\frac{\alpha q}{p}}  \sum_{j\in \mathbb{Z}} 2^{\frac{\alpha q}{p}j} e^{-q2^{\alpha j}} =C(\alpha,p,q)<+\infty,$$
since $2^{\frac{\alpha q}{p}j} e^{-q2^{\alpha j}} \underset{j \rightarrow -\infty}{\sim} 2^{\frac{\alpha q }{p}j}$ and $2^{\frac{\alpha q}{p}j} e^{-q2^{\alpha j}} \underset{j \rightarrow +\infty}{=} \mathcal{O}(e^{-j})$.
With these computations we can write now 
\begin{eqnarray*}
\|e^{-t|\xi|^\alpha} |\widehat{\vu}_0|\|_{L^{q,\infty}} &\leq & t^{-\frac{1}{p}}\|\vu_0\|_{ \dot{\mathscr{B}}^{-\frac{\alpha}{p},\infty}_{L^{q,\infty}_\xi}}\left( \displaystyle{\sum_{j\in \mathbb{Z}}} (t2^{j\alpha})^{\frac{q}{p}}\times  e^{-qt2^{j\alpha}}\right)^\frac{1}{q}\\
&\leq &C  t^{-\frac{1}{p}}\|\vu_0\|_{ \dot{\mathscr{B}}^{-\frac{\alpha}{p},\infty}_{L^{q,\infty}_\xi}}.
\end{eqnarray*}
Since the function $h:t\mapsto \|e^{-t|\xi|^\alpha}|\widehat{\vu}_0|\|_{L^{q,\infty}_\xi}$ is non increasing, it is equal to its own non increasing rearrangement, and taking the supremum of this estimate with respect to $t>0$, we find the quasi-norm $L^{p,\infty}_t$: 
$$\|e^{-t|\xi|^\alpha}|\widehat{\vu}_0|\|_{L^{p,\infty}_t(L^{q,\infty})}  \leq C \|\vu_0\|_{ \dot{\mathscr{B}}^{-\frac{\alpha}{p},\infty}_{L^{q,\infty}_\xi}},$$
from which we easily deduce the wished estimate
$$\|\mathfrak{p}_t\ast\vu_0\|_{L^{p,\infty}_t(L^{q,\infty}_\xi)}  \leq C \|\vu_0\|_{ \dot{\mathscr{B}}^{-\frac{\alpha}{p},\infty}_{L^{q,\infty}_\xi}}.$$
The same computations will lead to the control 
$$\|\mathfrak{p}_t\ast\theta_0\|_{L^{\mathfrak{p},\infty}_t(L^{\mathfrak{q},\infty}_\xi)}  \leq C \|\theta_0\|_{ \dot{\mathscr{B}}^{-\frac{\alpha}{\mathfrak{p}},\infty}_{L^{\mathfrak{q},\infty}_\xi}},$$
and we finally obtain the estimates (\ref{Estimation_DonneesInitiales}) for the initial data.
%%%%%%%%%%%%%%%%%%%%%%%%%%%%%%%%%%%%%%%%%%%%%%%%%%%
\item[(2)] Next, we consider the external force and we will proof the estimate 
$$\left\|\int_0^t\mathfrak{p}_{t-s}\ast\mathbb{P}(\vec{f})ds\right\|_{L^{p,\infty}_t(L^{q,\infty}_\xi)}\leq  C\|\vec{f}\|_{L^{{\bf p},\infty}_t(\mathcal{F}^{\mathfrak{s}}_{L^{{\bf q},\infty}_\xi})}.$$
This estimate will be a consequence of the following result.

%%%%%%%%%%%%%%%%%%%%%%%%%%%%%%%%%%%%%%%%%%%%%%%%%%%
\begin{Proposition}\label{prop_estm_force_red}
Let $d\geq 2$ and $0<\alpha\leq 2$ be fixed. Consider $1<p_1, p_2, q_1, q_2<+\infty$ and $0<\mathfrak{s}$ some real parameters such that we have the conditions $q_2<q_1$ and $p_1<p_2$, $0<\mathfrak{s}<d(\frac{1}{q_2}-\frac{1}{q_1})$ as well as the relationship
$\frac{1}{\alpha}\left(d(\frac{1}{q_2}-\frac{1}{q_1})-\mathfrak{s}\right)=\left(1+\frac{1}{p_2}-\frac{1}{p_1}\right)$. For $\sigma_0(D)$ a pseudo-differential operator with homogeneous symbol of degree $0$ smooth outside the origin, there exists a constant $0<\mathcal{C}<+\infty$ depending on all parameters and $\sigma_0$ such that for all $\vec{\varphi}:[0,+\infty[\times\mathbb{R}^d\longrightarrow \mathbb{R}^d$ with $\vec{\varphi} \in L^{p_1,\infty}_t(\mathcal{F}^{\mathfrak{s}}_{L^{q_1,\infty}_\xi})$, we have the estimate
$$\left \| \int_0^t\mathfrak{p}_{t-s}*\sigma_0(D)(\vec{\varphi})(s,\cdot)ds \right\|_{L^{p_2,\infty}_t(L^{q_2,\infty}_\xi)} \leq \mathcal{C}\| \vec{\varphi}\|_{L^{p_1,\infty}_t(\mathcal{F}^{\mathfrak{s}}_{L^{q_1,\infty}_\xi})}.$$
\end{Proposition}
%%%%%%%%%%%%%%%%%%%%%%%%%%%%%%%%%%%%%%%%%%%%%%%%%%%
{\bf Proof.} For a $t>0$ fixed, we start by considering the quantity 
$$\left|\mathcal{F}\left(\int_0^t\mathfrak{p}_{t-s}*\sigma_0(D)(\vec{\varphi})(s,\cdot)ds\right)(\xi)\right|\leq C\int_0^te^{-(t-s)|\xi|^\alpha}|\widehat{\vec{\varphi}}|(s, \xi)ds,$$
where we used the fact that the symbol $\sigma_0(D)$ is bounded in the Fourier level. Then applying the functional $\|\cdot\|_{{L^{q_2,\infty}_\xi}}$ to this expression we can write 
\begin{eqnarray*}
\left\| \int_0^t\mathfrak{p}_{t-s}*\sigma_0(D)(\vec{\varphi})(s,\cdot)ds \right\|_{L^{q_2,\infty}_\xi} &\leq &C\left\|\int_0^te^{-(t-s)|\xi|^\alpha}|\widehat{\vec{\varphi}}|(s, \cdot)ds \right\|_{L^{q_2,\infty}} \\
&\leq &C\int_0^t\left\|e^{-(t-s)|\xi|^\alpha}|\xi|^{-{\mathfrak{s}}}|\xi|^{\mathfrak{s}} |\widehat{\vec{\varphi}}|(s, \cdot)\right\|_{L^{q_2,\infty}}ds,
\end{eqnarray*}
for some parameter ${\mathfrak{s}}$ such that $0\leq {\mathfrak{s}}<d(\frac{1}{q_2}-\frac{1}{q_1})$. Now, by the H\"older inequality (\ref{Holder1}) in Lorentz spaces with $\frac{1}{q_2}=\frac{1}{r}+\frac{1}{q_1}$, with $1<r<+\infty$ and $1<q_2<q_1<+\infty$, we obtain 
$$\left\| \int_0^t\mathfrak{p}_{t-s}*\sigma_0(D)(\vec{\varphi})(s,\cdot)ds \right\|_{L^{q_2,\infty}_\xi} \leq C \int_0^t\left\|e^{-(t-s)|\xi|^\alpha}|\xi|^{-\mathfrak{s}}\right\|_{L^{r,\infty}}\left\||\xi|^\mathfrak{s} |\widehat{\vec{\varphi}}|(s, \cdot)\right\|_{L^{q_1,\infty}}ds.$$
At this point we remark that we have the control 
$$\left\|e^{-(t-s)|\xi|^\alpha}|\xi|^{-\mathfrak{s}}\right\|_{L^{r,\infty}}\leq C(t-s)^{-\frac{1}{\alpha}(\frac{d}{r}-\mathfrak{s})},$$
indeed, we can write $\left\|e^{-(t-s)|\xi|^\alpha}|\xi|^{-\mathfrak{s}}\right\|_{L^{r,\infty}}=(t-s)^{\frac{\mathfrak{s}}{\alpha}}\left\|e^{-|(t-s)^{\frac{1}{\alpha}}\xi|^\alpha}|(t-s)^{\frac{1}{\alpha}}\xi|^{-\mathfrak{s}}\right\|_{L^{r,\infty}}$ and by the homogeneous properties of the Lorentz space $L^{r,\infty}(\mathbb{R}^d)$ (recall the property (\ref{HomogeneiteLorentz}) above) we obtain the identities\\ 
$\left\|e^{-(t-s)|\xi|^\alpha}|\xi|^{-\mathfrak{s}}\right\|_{L^{r,\infty}}=(t-s)^{\frac{\mathfrak{s}}{\alpha}}\left\|e^{-|(t-s)^{\frac{1}{\alpha}}\xi|^\alpha}|(t-s)^{\frac{1}{\alpha}}\xi|^{-\mathfrak{s}}\right\|_{L^{r,\infty}}=(t-s)^{-\frac{1}{\alpha}(\frac{d}{r}-\mathfrak{s})}\left\|e^{-|\xi|^\alpha}|\xi|^{-\mathfrak{s}}\right\|_{L^{r,\infty}}$ and to finish we remark that $\left\|e^{-|\xi|^\alpha}|\xi|^{-\mathfrak{s}}\right\|_{L^{r,\infty}}\leq C\left\|e^{-|\xi|^\alpha}|\xi|^{-\mathfrak{s}}\right\|_{L^{r}}\leq C<+\infty$ since $0\leq \mathfrak{s}<\frac{d}{r}$.\\

We can thus write 
$$\left\| \int_0^t\mathfrak{p}_{t-s}*\sigma_0(D)(\vec{\varphi})(s,\cdot)ds \right\|_{L^{q_2,\infty}_\xi} \leq C \int_0^t(t-s)^{-\frac{1}{\alpha}(\frac{d}{r}-\mathfrak{s})}\left\||\xi|^\mathfrak{s} |\widehat{\vec{\varphi}}|(s, \cdot)\right\|_{L^{q_1,\infty}}ds.$$
By extending by $0$ the function $\widehat{\varphi}(s, \cdot)$ when $s<0$, we can write the previous integral in the following manner
\begin{eqnarray*}
\int_0^t(t-s)^{-\frac{1}{\alpha}(\frac{d}{r}-\mathfrak{s})}\left\||\xi|^\mathfrak{s} |\widehat{\vec{\varphi}}|(s, \cdot)\right\|_{L^{q_1,\infty}}ds&=&\int_{-\infty}^{+\infty}(t-s)^{-\frac{1}{\alpha}(\frac{d}{r}-\mathfrak{s})}\mathds{1}_{\{t-s>0\}}\left\||\xi|^\mathfrak{s} |\widehat{\varphi}|(s, \cdot)\right\|_{L^{q_1,\infty}}ds\\
&=&\left[ s^{-\frac{1}{\alpha}(\frac{d}{r}-\mathfrak{s})}\mathds{1}_{\{s>0\}}\right]* \left[ \||\xi|^{\mathfrak{s}}|\widehat{\vec{\varphi}}|(\cdot, \cdot)\|_{L^{q_1,\infty}}\right] (t),
\end{eqnarray*}
and we obtain the estimate 
$$\left\| \int_0^t\mathfrak{p}_{t-s}*\sigma_0(D)(\vec{\varphi})(s,\cdot)ds \right\|_{L^{q_2,\infty}_\xi} \leq C\left[ s^{-\frac{1}{\alpha}(\frac{d}{r}-\mathfrak{s})}\mathds{1}_{\{s>0\}}\right]* \left[ \||\xi|^{\mathfrak{s}}|\widehat{\vec{\varphi}}|(\cdot, \cdot)\|_{L^{q_1,\infty}}\right] (t).$$
We apply now the $L^{p_2,\infty}_t$ norm in the variable $t$ to the previous expression and we obtain 
$$\left\|\left\| \int_0^t\mathfrak{p}_{t-s}*\sigma_0(D)(\vec{\varphi})(s,\cdot)ds \right\|_{L^{q_2,\infty}_\xi}\right\|_{L^{p_2,\infty}_t} \leq C\left\|\left[ s^{-\frac{1}{\alpha}(\frac{d}{r}-\mathfrak{s})}\mathds{1}_{\{s>0\}}\right]* \left[ \||\xi|^{\mathfrak{s}}|\widehat{\vec{\varphi}}|(\cdot, \cdot)\|_{L^{q_1,\infty}}\right] (t)\right\|_{L^{p_2,\infty}_t}.$$
We can then apply the Young inequality for Lorentz spaces on $\mathbb{R}$, with parameters $1+\frac{1}{p_2}=\frac{1}{\kappa} + \frac{1}{p_1}$ (see the estimate (\ref{Young1}) above) and we have 
$$\left\| \int_0^t\mathfrak{p}_{t-s}*\sigma_0(D)(\vec{\varphi})(s,\cdot)ds \right\|_{L^{p_2,\infty}_t(L^{q_2,\infty}_\xi)}\leq C\left\| s^{-\frac{1}{\alpha}(\frac{d}{r}-\mathfrak{s})}\mathds{1}_{\{s>0\}}\right\|_{L^{\kappa, \infty}_t} \left\| \||\xi|^{\mathfrak{s}}|\widehat{\vec{\varphi}}|(\cdot, \cdot)\|_{L^{q_1,\infty}} \right\|_{L^{p_1,\infty}_t},$$
which, using the definition of the space $L^{p_1,\infty}_t(\mathcal{F}^{\mathfrak{s}}_{L^{q_1,\infty}_\xi})$ given in the expression (\ref{Def_NormFHL}), we can rewrite as follows
$$\left\| \int_0^t\mathfrak{p}_{t-s}*\sigma_0(D)(\vec{\varphi})(s,\cdot)ds \right\|_{L^{p_2,\infty}_t(L^{q_2,\infty}_\xi)}\leq C\left\| s^{-\frac{1}{\alpha}(\frac{d}{r}-\mathfrak{s})}\mathds{1}_{\{s>0\}}\right\|_{L^{\kappa, \infty}_t} \|\vec{\varphi}\|_{L^{p_1,\infty}_t(\mathcal{F}^{\mathfrak{s}}_{L^{q_1,\infty}_\xi})}.$$
We only need to prove now that the quantity is bounded $\left\| s^{-\frac{1}{\alpha}(\frac{d}{r}-\mathfrak{s})}\mathds{1}_{\{s>0\}}\right\|_{L^{\kappa, \infty}_t}$. Recalling that we have by hypothesis the relationship $\frac{1}{\alpha}\left(d(\frac{1}{q_2}-\frac{1}{q_1})-\mathfrak{s}\right)=\left(1+\frac{1}{p_2}-\frac{1}{p_1}\right)$ as well as $\frac{1}{r}=\frac{1}{q_2}-\frac{1}{q_1}$ and $\frac{1}{\kappa}=1+\frac{1}{p_2}-\frac{1}{p_1}$, we have that $\frac{1}{\alpha}(\frac{d}{r}-\mathfrak{s})=\frac{1}{\kappa}$ and we obtain $\left\| s^{-\frac{1}{\alpha}(\frac{d}{r}-\mathfrak{s})}\mathds{1}_{\{s>0\}}\right\|_{L^{\kappa, \infty}_t}=\left\| s^{-\frac{1}{\kappa}}\mathds{1}_{\{s>0\}}\right\|_{L^{\kappa, \infty}_t}<+\infty$. We have thus proven the control 
$$\left\| \int_0^t\mathfrak{p}_{t-s}*\sigma_0(D)(\vec{\varphi})(s,\cdot)ds \right\|_{L^{p_2,\infty}_t(L^{q_2,\infty}_\xi)}\leq C\|\vec{\varphi}\|_{L^{p_1,\infty}_t(\mathcal{F}^{\mathfrak{s}}_{L^{q_1,\infty}_\xi})},$$
which is the wished estimate. \hfill $\blacksquare$\\
%%%%%%%%%%%%%%%%%%%%%%%%%%%%%%%%%%%%%%%%%%%%%%%%%%%

We are going to exploit this result. Indeed, if we set $p=p_2$, $q=q_2$ and ${\bf p}=p_1$, ${\bf q}=q_1$, since we have by hypothesis the relationship $\frac{1}{\alpha}\left(d(\frac{1}{q}-\frac{1}{\bf q})-\mathfrak{s}\right)=\left(1+\frac{1}{p}-\frac{1}{\bf p}\right)$, which is equivalent to the condition $\frac{1}{\bf p}-\frac{d}{\alpha {\bf q}}-\mathfrak{s}=1+\frac{1}{p}-\frac{d}{\alpha q}$, we may apply the previous proposition (due to the properties of the Leray projector in the Fourier variable, see \cite[Chapitre 5]{Chamorro_Livre}) to obtain the estimate:
$$\left\|\int_0^t\mathfrak{p}_{t-s}\ast\mathbb{P}(\vec{f})ds\right\|_{L^{p,\infty}_t(L^{q,\infty}_\xi)}\leq  C\|\vec{f}\|_{L^{{\bf p},\infty}_t(\mathcal{F}^{\mathfrak{s}}_{L^{{\bf q},\infty}_\xi})}.$$

\item[(3)] For the linear term we can apply the Proposition \ref{prop_estm_force_red} to obtain
$$\left\|\int_0^t\mathfrak{p}_{t-s}\ast\mathbb{P}(\theta \vec{e}_d)ds\right\|_{L^{p,\infty}_t(L^{q,\infty}_\xi)}\leq  C \|\theta\|_{L^{\mathfrak{p},\infty}_t(L^{\mathfrak{q},\infty}_\xi)},$$
since we have the relationship $\frac{d}{\alpha}(\frac{1}{q}-\frac{1}{\mathfrak{q}})=\left(1+\frac{1}{p}-\frac{1}{\mathfrak{p}}\right)$, which can be easily deduced from the condition $\frac{1}{\mathfrak{p}}-\frac{d}{\alpha \mathfrak{q}}=1+\frac{1}{p}-\frac{d}{\alpha q}$.

\item[(4)] For the bilinear term we need to study the following estimates
$$\left\|\int_0^t\mathfrak{p}_{t-s}\ast\mathbb{P}\left(\mathrm{div}(\vec{u}\otimes \vec{u})\right)ds\right\|_{L^{p,\infty}_t(L^{q,\infty}_\xi)}\leq C\|\vu\|_{L^{p,\infty}_t(L^{q,\infty}_\xi)}\|\vu\|_{L^{p,\infty}_t(L^{q,\infty}_\xi)},$$
and
$$\left\|\int_0^t \mathfrak{p}_{t-s}\ast \mathrm{div}(\theta \vec{u})ds\right\|_{L^{\mathfrak{p},\infty}_t(L^{\mathfrak{q},\infty}_\xi)}\leq C\|\vu\|_{L^{p,\infty}_t(L^{q,\infty}_\xi)}\|\theta\|_{L^{\mathfrak{p},\infty}_t(L^{\mathfrak{q},\infty}_\xi)}.$$

We will study these estimates with the following generic result.
%%%%%%%%%%%%%%%%%%%%%%%%%%%%%%%%%%%%%%%%%%%%%%%%%%%
\begin{Proposition}[bilinear estimate]\label{Prop_generalquadratic}
Let $d\geq2 $ and $1<\alpha<2$ be fixed. Consider now $1<p_1,p_2, p_3<+\infty$ and $1<q_1,q_2, q_3<+\infty$ be a set of parameters such that we have the relationship 
$$\frac{d}{\alpha}\left(1+\frac{1}{q_3}-(\frac{1}{q_1}+\frac{1}{q_2})\right)+\frac{1}{\alpha}=1+\frac{1}{p_3}-(\frac{1}{p_1}+\frac{1}{p_2}),$$ 
with $1<\frac{1}{q_1}+\frac{1}{q_2}<1+\frac{1}{q_3}$ and $\frac{1}{p_3}<\frac{1}{p_1}+\frac{1}{p_2}<1$, then there exists a constant $C>0$ depending on all parameters such that for all $f \in L^{p_1,\infty}_t (L^{q_1,\infty}_\xi)$ and $g \in L^{p_2,\infty}_t(L^{q_2,\infty}_\xi)$, we have the inequality 
$$\left\|\int_0^te^{-(t-s)|\xi|^\alpha}|\xi| \times \left[|\widehat{f}|*|\widehat{g}| \right] ds\right\|_{L^{p_3,\infty}_t(L^{q_3,\infty})} \leq  C \left\| \widehat{f}\right\|_{L^{p_1,\infty}_t(L^{q_1,\infty})}  \left\| \widehat{g}\right\|_{L^{p_2,\infty}_t(L^{q_2,\infty})}.$$
 \end{Proposition}
%%%%%%%%%%%%%%%%%%%%%%%%%%%%%%%%%%%%%%%%%%%%%%%%%%%
{\bf Proof.} We start by writing 
$$ \left|\int_0^te^{-(t-s)|\xi|^\alpha}|\xi| \times \left[|\widehat{f}|*|\widehat{g}|(s,\xi) \right] ds\right|\leq \displaystyle{\int_0^t} e^{-(t-s)|\xi|^\alpha}|\xi|\times \left[|\widehat{f}|*|\widehat{g}|(s,\xi)\right] ds,$$
and for a fixed time $t>0$, we take the $L^{q_3,\infty}$ norm in order to obtain
$$\left\|\int_0^te^{-(t-s)|\xi|^\alpha}|\xi| \times \left[|\widehat{f}|*|\widehat{g}|(s,\xi) \right] ds\right\|_{L^{q_3,\infty}} \leq \displaystyle{\int_0^t} \left\|e^{-(t-s)|\xi|^\alpha} |\xi| \times \left[|\widehat{f}|*|\widehat{g}|(s,\xi)\right] \right\|_{L^{q_3,\infty}}ds,$$
now, by the H\"older inequality (\ref{Holder1}) in Lorentz spaces with $\frac{1}{q_3}=\frac{1}{r}+\frac{1}{\rho}$  where $\frac{1}{\rho}=\frac{1}{q_1}+\frac{1}{q_2}-1$, we have
$$\left\|\int_0^te^{-(t-s)|\xi|^\alpha}|\xi| \times \left[|\widehat{f}|*|\widehat{g}|(s,\xi) \right] ds\right\|_{L^{q_3,\infty}} \leq \displaystyle{\int_0^t} \left\|e^{-(t-s)|\xi|^\alpha} |\xi| \right\|_{L^{r,\infty}} \left\|\left[|\widehat{f}|*|\widehat{g}|(s,\xi)\right] \right\|_{L^{\rho,\infty}}ds.$$
Note that by homogeneity (recall the property (\ref{HomogeneiteLorentz}) above for the Lorentz spaces) we easily get the identity $\left\|e^{-(t-s)|\xi|^\alpha} |\xi|\right\|_{L^{r,\infty}} = \left\|e^{-|\xi|^\alpha} |\xi|\right\|_{L^{r,\infty}}(t-s)^{-\frac{1}{\alpha} (\frac{d}{r}+1)}$, and since we have $\left\|e^{-|\xi|^\alpha} |\xi|\right\|_{L^{r,\infty}}\leq \left\|e^{-|\xi|^\alpha} |\xi|\right\|_{L^{r}}<+\infty$, we can write 
$$\left\|\int_0^te^{-(t-s)|\xi|^\alpha}|\xi| \times \left[|\widehat{f}|*|\widehat{g}|(s,\xi) \right] ds\right\|_{L^{q_3,\infty}} \leq C\int_0^t (t-s)^{-\frac{1}{\alpha} (\frac{d}{r}+1)}\left\|\left[|\widehat{f}|*|\widehat{g}|(s,\xi)\right] \right\|_{L^{\rho,\infty}}ds.$$
Applying again the Young inequality (\ref{Young1}) for Lorentz spaces with $1+\frac{1}{\rho}=\frac{1}{q_1}+\frac{1}{q_2}$, we obtain
$$\left\|\int_0^te^{-(t-s)|\xi|^\alpha}|\xi| \times \left[|\widehat{f}|*|\widehat{g}|(s,\xi) \right] ds\right\|_{L^{q_3,\infty}} \leq C\int_0^t (t-s)^{-\frac{1}{\alpha} (\frac{d}{r}+1)}\left\|\widehat{f}(s,\cdot)\right\|_{L^{q_1,\infty}}\left\|\widehat{g}(s,\cdot)\right\|_{L^{q_2,\infty}}ds.$$
%%%%%%%%%%%%%%%%%%%%%%%%%%%%%%%%%%%%%%%%%%%%%%%%%%%
We extend the functions $\widehat{f}(s,\cdot)$ and $\widehat{g}(s,\cdot)$ by $0$ when $s<0$ and we can thus consider the previous integral in the time variable as a convolution, so we can write 
\begin{eqnarray*}
\left\|\int_0^te^{-(t-s)|\xi|^\alpha}|\xi| \times \left[|\widehat{f}|*|\widehat{g}|(s,\xi) \right] ds\right\|_{L^{q_3,\infty}} \leq C\int_{-\infty}^{+\infty} (t-s)^{-\frac{1}{\alpha} (\frac{d}{r}+1)}\mathds{1}_{\{t-s>0\}}\qquad \\
\times\left\|\widehat{f}(s,\cdot)\right\|_{L^{q_1,\infty}_\xi}\left\|\widehat{g}(s,\cdot)\right\|_{L^{q_2,\infty}_\xi}ds\\
\leq  C\left(\mathds{1}_{\{s>0\}} s^{-\frac{1}{\alpha} (\frac{d}{r}+1)}\right)\ast \left(\left\|\widehat{f}(\cdot,\cdot)\right\|_{L^{q_1,\infty}}\left\|\widehat{g}(\cdot,\cdot)\right\|_{L^{q_2,\infty}}\right)(t).
\end{eqnarray*}
Now, taking the $L^{p_3,\infty}$ norm in the expression above we have 
\begin{eqnarray*}
\left\|\left\|\int_0^te^{-(t-s)|\xi|^\alpha}|\xi| \times \left[|\widehat{f}|*|\widehat{g}| \right] ds\right\|_{L^{q_3,\infty}}\right\|_{L^{p_3,\infty}_t}=\left\|\int_0^te^{-(t-s)|\xi|^\alpha}|\xi| \times \left[|\widehat{f}|*|\widehat{g}| \right] ds\right\|_{L^{p_3,\infty}_t(L^{q_3,\infty})}\\ 
\leq  C\left\|\left(\mathds{1}_{\{s>0\}} s^{-\frac{1}{\alpha} (\frac{d}{r}+1)}\right)\ast \left(\left\|\widehat{f}(\cdot,\cdot)\right\|_{L^{q_1,\infty}}\left\|\widehat{g}(\cdot,\cdot)\right\|_{L^{q_2,\infty}}\right)\right\|_{L^{p_3,\infty}_t}.
\end{eqnarray*}
At this point, we apply the Young inequality (\ref{Young1}) for Lorentz spaces (in the time variable) with $1+\frac{1}{p_3}=\frac{1}{\mathfrak{r}}+\frac{1}{\beta}$ where $\frac{1}{\beta}=\frac{1}{p_1}+\frac{1}{p_2}$, to obtain 
\begin{eqnarray*}
\left\|\int_0^te^{-(t-s)|\xi|^\alpha}|\xi| \times \left[|\widehat{f}|*|\widehat{g}| \right] ds\right\|_{L^{p_3,\infty}_t(L^{q_3,\infty})} &\leq & C\left\|\mathds{1}_{\{s>0\}} s^{-\frac{1}{\alpha} (\frac{d}{r}+1)}\right\|_{L^{\mathfrak{r},\infty}_t}\\
&&\times \left\| \left\|\widehat{f}(\cdot,\cdot)\right\|_{L^{q_1,\infty}}\left\|\widehat{g}(\cdot,\cdot)\right\|_{L^{q_2,\infty}}\right\|_{L^{\beta,\infty}_t}.
\end{eqnarray*}
Recall that we have the relationship $\frac{1}{\alpha} (\frac{d}{r}+1)=\frac{1}{\mathfrak{r}}$, so we have 
$\left\|\mathds{1}_{\{s>0\}} s^{-\frac{1}{\alpha} (\frac{d}{r}+1)}\right\|_{L^{\mathfrak{r},\infty}_t}<+\infty$, and we obtain the estimate 
$$\left\|\int_0^te^{-(t-s)|\xi|^\alpha}|\xi| \times \left[|\widehat{f}|*|\widehat{g}| \right] ds\right\|_{L^{p_3,\infty}_t(L^{q_3,\infty})} \leq  C \left\| \left\|\widehat{f}(\cdot,\cdot)\right\|_{L^{q_1,\infty}}\left\|\widehat{g}(\cdot,\cdot)\right\|_{L^{q_2,\infty}}\right\|_{L^{\beta,\infty}_t}.$$
Finally, we use Hölder's inequality (\ref{Holder1}) for Lorentz spaces with $\frac{1}{\beta}=\frac{1}{p_1}+\frac{1}{p_2}$ to get
$$\left\|\int_0^te^{-(t-s)|\xi|^\alpha}|\xi| \times \left[|\widehat{f}|*|\widehat{g}| \right] ds\right\|_{L^{p_3,\infty}_t(L^{q_3,\infty})} \leq  C \left\| \left\|\widehat{f}(\cdot,\cdot)\right\|_{L^{q_1,\infty}}\right\|_{L^{p_1,\infty}_t} \left\|\left\|\widehat{g}(\cdot,\cdot)\right\|_{L^{q_2,\infty}}\right\|_{L^{p_2,\infty}_t},$$
which we rewrite as follows: 
$$\left\|\int_0^te^{-(t-s)|\xi|^\alpha}|\xi| \times \left[|\widehat{f}|*|\widehat{g}| \right] ds\right\|_{L^{p_3,\infty}_t(L^{q_3,\infty})} \leq  C \left\| \widehat{f}\right\|_{L^{p_1,\infty}_t(L^{q_1,\infty})}  \left\| \widehat{g}\right\|_{L^{p_2,\infty}_t(L^{q_2,\infty})}, $$
which is the wished inequality and Proposition \ref{Prop_generalquadratic} is now proven. \hfill $\blacksquare$

With result at hand, the estimates for the bilinear terms are easy to obtain, indeed:
\begin{itemize}
\item[$\bullet$] if $p_1=p_2=p_3$ and if $q_1=q_2=q_3$, the relationship $\frac{d}{\alpha}\left(1+\frac{1}{q_3}-(\frac{1}{q_1}+\frac{1}{q_2})\right)+\frac{1}{\alpha}=1+\frac{1}{p_3}-(\frac{1}{p_1}+\frac{1}{p_2})$ becomes $\frac{d}{\alpha}\left(1-\frac{1}{q}\right)+\frac{1}{\alpha}=1-\frac{1}{p}$, which can be rewritten as $\frac{\alpha}{p}+\frac{d}{q'}=\alpha-1$, which is the condition stated as hypothesis in the Theorem \ref{thm_boussinesq_simple}. We write then 
$$\left\|\int_0^t\mathfrak{p}_{t-s}\ast \mathbb{P}(div(\vu\otimes\vu))ds\right\|_{L^{p,\infty}_t(L^{q,\infty}_\xi)} \leq C\left\|\int_0^te^{-(t-s)|\xi|^\alpha}|\xi| \times \left[|\widehat{\vu}|*|\widehat{\vu}| \right] ds\right\|_{L^{p,\infty}_t(L^{q,\infty})},$$
thus by applying the Proposition \ref{Prop_generalquadratic} to the right-hand side of the previous estimate, we obtain the inequality
$$\left\|\int_0^te^{-(t-s)|\xi|^\alpha}|\xi| \times \left[|\widehat{\vu}|*|\widehat{\vu}| \right] ds\right\|_{L^{p,\infty}_t(L^{q,\infty})}\leq C\| \widehat{\vu}\|_{L^{p,\infty}_t(L^{q,\infty})}  \| \widehat{\vu}\|_{L^{p,\infty}_t(L^{q,\infty})},$$
from which we deduce 
$$\left\|\int_0^t\mathfrak{p}_{t-s}\ast \mathbb{P}(div(\vu\otimes\vu))ds\right\|_{L^{p,\infty}_t(L^{q,\infty}_\xi)} \leq  C\| \vu\|_{L^{p,\infty}_t(L^{q,\infty}_\xi)}  \| \vu\|_{L^{p,\infty}_t(L^{q,\infty}_\xi)},$$
which gives us the boundedness of the first term of the bilinear term of the equation (\ref{Estimation_Bilineaire}).
\item[$\bullet$] For the second bilinear term, we have
$$\left\|\int_0^t \mathfrak{p}_{t-s}\ast \mathrm{div}(\theta \vec{u})ds\right\|_{L^{\mathfrak{p},\infty}_t(L^{\mathfrak{q},\infty}_\xi)}\leq \left\|\int_0^te^{-(t-s)|\xi|^\alpha}|\xi| \times \left[|\widehat{\vu}|*|\widehat{\theta}| \right] ds\right\|_{L^{\mathfrak{p},\infty}_t(L^{\mathfrak{q},\infty})}.$$
In order to apply the Proposition \ref{Prop_generalquadratic} we set $p_1=p$, $q_1=q$ and $p_2=p_3=\mathfrak{p}$, $q_2=q_3=\mathfrak{q}$, so the condition $\frac{d}{\alpha}\left(1+\frac{1}{q_3}-(\frac{1}{q_1}+\frac{1}{q_2})\right)+\frac{1}{\alpha}=1+\frac{1}{p_3}-(\frac{1}{p_1}+\frac{1}{p_2})$ becomes once more time $\frac{\alpha}{p}+\frac{d}{q'}=\alpha-1$, which is given as hypothesis and we can write
$$ \left\|\int_0^te^{-(t-s)|\xi|^\alpha}|\xi| \times \left[|\widehat{\vu}|*|\widehat{\theta}| \right] ds\right\|_{L^{\mathfrak{p},\infty}_t(L^{\mathfrak{q},\infty})}\leq C\|\widehat{\vu}\|_{L^{p,\infty}_t(L^{q,\infty})}  \| \widehat{\theta}\|_{L^{\mathfrak{p},\infty}_t(L^{\mathfrak{q},\infty})},$$
from which we easily deduce the estimate 
$$\left\|\int_0^t \mathfrak{p}_{t-s}\ast \mathrm{div}(\theta \vec{u})ds\right\|_{L^{\mathfrak{p},\infty}_t(L^{\mathfrak{q},\infty}_\xi)}\leq C\|\vu\|_{L^{p,\infty}_t(L^{q,\infty}_\xi)}  \|\theta\|_{L^{\mathfrak{p},\infty}_t(L^{\mathfrak{q},\infty}_\xi)},$$
and this ends the boundedness of the bilinear term. 
\end{itemize}
\end{itemize}
With all these estimates, if the quantity $\|\vu_0\|_{\dot{\mathscr{B}}^{-\frac{\alpha}{p}, \infty}_{L^{q,\infty}_\xi}}+\|\theta_0\|_{\dot{\mathscr{B}}^{-\frac{\alpha}{\mathfrak{p}}, \infty}_{L^{\mathfrak{q},\infty}_\xi}}+\|\vec{f}\|_{L^{{\bf p},\infty}_t(\mathcal{F}^{\mathfrak{s}}_{L^{{\bf q},\infty}_\xi})}$ is small enough, then we can apply the Lemma \ref{Lemme_Banach_Picard} to obtain a mild solution $(\vu, \theta)$ to the problem (\ref{eq_NSBalpha}) such that $(\vu, \theta)\in L^{p,\infty}_t(L^{q,\infty}_\xi)\times L^{\mathfrak{p},\infty}_t(L^{\mathfrak{q},\infty}_\xi)$ and this ends the proof of the Theorem \ref{thm_boussinesq_simple}. $\hfill \blacksquare$
%%%%%%%%%%%%%%%%%%%%%%%%%%%%%%%%%%%%%%%%%%%%%%%%%%%
\section{Proof of the Theorem \ref{thm_boussinesq_L1_alpha}}
The proof of this result follows the same main ideas displayed in the previous section and share some of the techniques presented above, however, some details must be made explicit in a careful manner. Indeed, by considering again the equation $\vvU=\vvU_0+\vvF+L(\vvU)+B(\vvU,\vvU)$ also given in (\ref{Equation_PointFixe}), with the terms $\vvU=\left[\begin{array}{c}\vu\\ \theta\end{array}\right]$, $\vvU_0 = \left[\begin{array}{c}
\mathfrak{p}_t\ast \vec{u_0} \\\mathfrak{p}_t\ast \theta_0\end{array}\right]$, $\vvF = \left[\begin{array}{c} \displaystyle{\int_0^t} \mathfrak{p}_{t-s}\ast\mathbb{P}(\vec{f})ds\\
0\end{array}\right]$,  $L(\vvU)=\left[\begin{array}{c} \displaystyle{\int_0^t} \mathfrak{p}_{t-s}\ast\mathbb{P}(\theta \vec{e}_d)ds\\  0\end{array}\right]$ and $B(\vvU,\vvU)=\left[\begin{array}{c} -\displaystyle{\int_0^t} \mathfrak{p}_{t-s}\ast\mathbb{P}\left(\mathrm{div}(\vec{u}\otimes \vec{u})\right)ds\\    -\displaystyle{\int_0^t} \mathfrak{p}_{t-s}\ast \mathrm{div}(\theta \vec{u})ds\end{array}\right]$, and if we fix the following resolution space $E=L^{\frac{\alpha}{\alpha-1},\infty}_t(L^1_\xi)\times L^{\mathfrak{p},\infty}_t(L^{\mathfrak{q},\infty}_\xi)$ endowed with the norm $\|\vvU\|_{E}=\left\|\left[\begin{array}{c}\vu\\ \theta\end{array}\right]\right\|_{E}=\|\vu\|_{L^{\frac{\alpha}{\alpha-1},\infty}_t(L^1_\xi)}+\mathfrak{C}\|\theta\|_{L^{\mathfrak{p},\infty}_t(L^{\mathfrak{q},\infty}_\xi)}$, then we only need to prove the following estimates:

\begin{itemize}
\item[(1)] For the initial data $\vvU_0$:
\begin{eqnarray}
\|\vvU_0\|_{E}&=&\|\mathfrak{p}_t\ast\vu_0\|_{L^{\frac{\alpha}{\alpha-1},\infty}_t(L^1_\xi)}+\mathfrak{C}\|\mathfrak{p}_t\ast\theta_0\|_{L^{\mathfrak{p},\infty}_t(L^{\mathfrak{q},\infty}_\xi)}\notag\\
&\leq & C\|\vu_0\|_{\dot{\mathscr{B}}^{-(\alpha-1), \infty}_{L^{1}_\xi}}+C\|\theta_0\|_{\dot{\mathscr{B}}^{-\frac{\alpha}{\mathfrak{p}},\infty}_{L^{\mathfrak{q},\infty}_\xi}}.\label{Estimation_DonneesInitialesAlpha}
\end{eqnarray}
\item[(2)] For the external force $\vvF$:
\begin{equation}
\|\vvF\|_{E} =\left\|\int_0^t\mathfrak{p}_{t-s}\ast\mathbb{P}(\vec{f})ds\right\|_{L^{\frac{\alpha}{\alpha-1},\infty}_t(L^1_\xi)}\leq C\|\vec{f}\|_{L^{{\bf p},\infty}_t(\mathcal{F}^{s}_{L^{{\bf q},\infty}_\xi})}.\label{Estimation_ForceAlpha}
\end{equation}
\item[(3)] For the Linear term $L(\vvU)$
\begin{equation}\label{Estimation_LineaireAlpha}
\|L(\vvU)\|_{E}=\left\|\int_0^t\mathfrak{p}_{t-s}\ast\mathbb{P}(\theta \vec{e}_d)ds\right\|_{L^{\frac{\alpha}{\alpha-1},\infty}_t(L^1_\xi)}\leq  C \|\theta\|_{L^{\mathfrak{p},\infty}_t(L^{\mathfrak{q},\infty}_\xi)}.
\end{equation}
\item[(4)] For the Bilinear term $B(\vvU,\vvU)$
\begin{eqnarray}
\|B(\vvU,\vvU)\|_{E}=\left\|\int_0^t\mathfrak{p}_{t-s}\ast\mathbb{P}\left(\mathrm{div}(\vec{u}\otimes \vec{u})\right)ds\right\|_{L^{\frac{\alpha}{\alpha-1},\infty}_t(L^1_\xi)}+\mathfrak{C}\left\|\int_0^t \mathfrak{p}_{t-s}\ast \mathrm{div}(\theta \vec{u})ds\right\|_{L^{\mathfrak{p},\infty}_t(L^{\mathfrak{q},\infty}_\xi)}\notag\\
\leq  C\|\vu\|_{L^{\frac{\alpha}{\alpha-1},\infty}_t(L^1_\xi)}\|\vu\|_{L^{\frac{\alpha}{\alpha-1},\infty}_t(L^1_\xi)}+C\|\vu\|_{L^{\frac{\alpha}{\alpha-1},\infty}_t(L^1_\xi)}\|\theta\|_{L^{\mathfrak{p},\infty}_t(L^{\mathfrak{q},\infty}_\xi)}.\qquad\qquad\label{Estimation_BilineaireAlpha}
\end{eqnarray}
\end{itemize}
As before, the estimates (\ref{Estimation_DonneesInitialesAlpha}), (\ref{Estimation_ForceAlpha}), (\ref{Estimation_LineaireAlpha}) and (\ref{Estimation_BilineaireAlpha}) will be treated separately.\\

\begin{itemize}
\item[(1)] For the initial data $\vu_0$ and $\theta_0$ we will show that we have the controls 
$$\|\mathfrak{p}_t\ast\vu_0\|_{L^{\frac{\alpha}{\alpha-1},\infty}_t(L^1_\xi)} \leq C \|\vu_0\|_{\dot{\mathscr{B}}^{-(\alpha-1), \infty}_{L^{1}_\xi}}\quad \mbox{and}\quad \|\mathfrak{p}_t\ast\theta_0\|_{ L^{\mathfrak{p},\infty}_t(L^{\mathfrak{q},\infty}_\xi)}\leq C\|\theta_0\|_{\dot{\mathscr{B}}^{-\frac{\alpha}{\mathfrak{p}},\infty}_{L^{\mathfrak{q},\infty}_\xi}}.$$
The second estimate was already studied in the previous theorem so we only focus here in the first one and we start by writing
$$\|\mathfrak{p}_t\ast\vu_0\|_{L^1_\xi}=\| e^{-t|\xi|^\alpha}\widehat{\vu}_0(\cdot)\|_{L^1}=\sum_{j\in \mathbb{Z}}\int_{\{2^j\leq |\xi|\leq 2^{j+1}\}}e^{-t|\xi|^\alpha}|\widehat{\vu}_0(\xi)|d\xi.$$
Since over the sets $\{2^j\leq |\xi|\leq 2^{j+1}\}$ we have $e^{-t|\xi|^\alpha}\leq e^{-t2^{j\alpha}}$, we have
$$\|\mathfrak{p}_t\ast\vu_0\|_{L^1_\xi}\leq \sum_{j\in \mathbb{Z}}e^{-t2^{j\alpha}}\|\mathds{1}_{\{2^j\leq |\xi|\leq 2^{j+1}\}}\widehat{\vu}_0(\cdot)\|_{L^1}= \sum_{j\in \mathbb{Z}}e^{-t2^{j\alpha}}2^{(\alpha-1)j}\times2^{-(\alpha-1)j}\|\mathds{1}_{\{2^j\leq |\xi|\leq 2^{j+1}\}}\widehat{\vu}_0(\cdot)\|_{L^1},$$
and using the definition of the norm $\|\vu_0\|_{\dot{\mathscr{B}}^{-(\alpha-1), \infty}_{L^{1}_\xi}}=\underset{j\in \mathbb{Z}}{\sup}\, 2^{-(\alpha-1)j}\|\mathds{1}_{\{2^j\leq |\xi|\leq 2^{j+1}\}}\widehat{\vu}_0(\cdot)\|_{L^1}$ (recall the expression (\ref{NormFBL})), we obtain
$$\|\mathfrak{p}_t\ast\vu_0\|_{L^1_\xi}\leq  \|\vu_0\|_{\dot{\mathscr{B}}^{-(\alpha-1), \infty}_{L^{1}_\xi}} \sum_{j\in \mathbb{Z}}e^{-t2^{j\alpha}}2^{(\alpha-1)j}.$$
Note now that the previous sum can be rewritten as follows
$$\|\mathfrak{p}_t\ast\vu_0\|_{L^1_\xi}\leq  t^{-\frac{\alpha-1}{\alpha}}\|\vu_0\|_{\dot{\mathscr{B}}^{-(\alpha-1), \infty}_{L^{1}_\xi}} \sum_{j\in \mathbb{Z}}e^{-t2^{j\alpha}}(t2^{j\alpha})^{\frac{(\alpha-1)}{\alpha}},$$
and we can easily seen that this sum is convergent (regardless of the time variable) so we have the estimate
$$  t^{\frac{\alpha-1}{\alpha}}\|\mathfrak{p}_t\ast\vu_0\|_{L^1_\xi}\leq C\|\vu_0\|_{\dot{\mathscr{B}}^{-(\alpha-1), \infty}_{L^{1}_\xi}}.$$
Noting that the function $h:t\mapsto \|\mathfrak{p}_t\ast\vu_0\|_{L^1_\xi}=\|e^{-t|\xi|^\alpha}|\widehat{\vu}_0|\|_{L^{1}}$ is non increasing, it is equal to its own non increasing rearrangement, and thus taking the supremum with respect to $t>0$ we have (recall the definition of the Lorentz norm given in the expression (\ref{Lorentzinfty}) above):
$$\|\mathfrak{p}_t\ast\vu_0\|_{L^{\frac{\alpha}{\alpha-1},\infty}_t(L^{1}_\xi)}=\underset{t>0}{\sup}\; t^{\frac{\alpha-1}{\alpha}}\|\mathfrak{p}_t\ast\vu_0\|_{L^1_\xi} \leq C\|\vu_0\|_{\dot{\mathscr{B}}^{-(\alpha-1), \infty}_{L^{1}_\xi}},$$
which is the announced control for $\vu_0$.
%%%%%%%%%%%%%%%%%%%%%%%%%%%%%%%%%%%%%%%%%%%%%%%%%%%
\item[(2)] The estimate needed for the external force $\vf$ will be a consequence of a variation of the  Proposition \ref{prop_estm_force_red}. We give the details here for the sake of completeness: we thus write
$$\left\|\int_0^t\mathfrak{p}_{t-s}\ast\mathbb{P}(\vec{f})ds\right\|_{L^1_\xi}\leq C\left\|\int_0^te^{-(t-s)|\xi|^\alpha}|\widehat{\vf}|(s, \cdot)ds \right\|_{L^{1}} \leq C\int_0^t\left\|e^{-(t-s)|\xi|^\alpha}|\xi|^{-{\mathfrak{s}}}|\xi|^{\mathfrak{s}} |\widehat{\vf}|(s, \cdot)\right\|_{L^{1}}ds,$$
with $0\leq {\mathfrak{s}}<\frac{d}{{\bf q}'}$. By the H\"older inequality (\ref{Holder2}) in Lorentz spaces with $1=\frac{1}{\bf q}+\frac{1}{{\bf q}'}$ we have 
$$\left\|\int_0^t\mathfrak{p}_{t-s}\ast\mathbb{P}(\vec{f})ds\right\|_{L^1_\xi} \leq C\int_0^t\left\|e^{-(t-s)|\xi|^\alpha}|\xi|^{-{\mathfrak{s}}}\right\|_{L^{{\bf q}',1}}\left\||\xi|^{\mathfrak{s}} |\widehat{\vf}|(s, \cdot)\right\|_{L^{{\bf q},\infty}}ds,$$
note that we have $\left\|e^{-(t-s)|\xi|^\alpha}|\xi|^{-\mathfrak{s}}\right\|_{L^{{\bf q}',1}}\leq C(t-s)^{-\frac{1}{\alpha}(\frac{d}{{\bf q}'}-\mathfrak{s})}$. Extending by $0$ the function $\widehat{\vf}(s, \cdot)$ when $s<0$, we can write
\begin{eqnarray*}
\int_0^t(t-s)^{-\frac{1}{\alpha}(\frac{d}{{\bf q}'}-\mathfrak{s})}\left\||\xi|^\mathfrak{s} |\widehat{\vf}|(s, \cdot)\right\|_{L^{{\bf q},\infty}}ds&=&\int_{-\infty}^{+\infty}(t-s)^{-\frac{1}{\alpha}(\frac{d}{{\bf q}'}-\mathfrak{s})}\mathds{1}_{\{t-s>0\}}\left\||\xi|^\mathfrak{s} |\widehat{\vf}|(s, \cdot)\right\|_{L^{{\bf q},\infty}}ds\\
&=&\left[ s^{-\frac{1}{\alpha}(\frac{d}{{\bf q}'}-\mathfrak{s})}\mathds{1}_{\{s>0\}}\right]* \left[ \||\xi|^{\mathfrak{s}}|\widehat{\vf}|(\cdot, \cdot)\|_{L^{{\bf q},\infty}_\xi}\right] (t),
\end{eqnarray*}
from which we obtain the control 
$$\left\|\int_0^t\mathfrak{p}_{t-s}\ast\mathbb{P}(\vec{f})ds\right\|_{L^1_\xi} \leq C\left[ s^{-\frac{1}{\alpha}(\frac{d}{{\bf q}'}-\mathfrak{s})}\mathds{1}_{\{s>0\}}\right]* \left[ \||\xi|^{\mathfrak{s}}|\widehat{\vf}|(\cdot, \cdot)\|_{L^{{\bf q},\infty}}\right] (t).$$
We apply now the $L^{\frac{\alpha}{\alpha-1},\infty}_t$ norm in the variable $t$ to the previous expression: 
$$\left\|\left\| \int_0^t\mathfrak{p}_{t-s}\ast\mathbb{P}(\vec{f})ds\right\|_{L^{1}_\xi}\right\|_{L^{\frac{\alpha}{\alpha-1},\infty}_t} \leq C\left\|\left[ s^{-\frac{1}{\alpha}(\frac{d}{{\bf q}'}-\mathfrak{s})}\mathds{1}_{\{s>0\}}\right]* \left[ \||\xi|^{\mathfrak{s}}|\widehat{\vf}|(\cdot, \cdot)\|_{L^{{\bf q},\infty}}\right] (t)\right\|_{L^{\frac{\alpha}{\alpha-1},\infty}_t}.$$
If we apply the Young inequality (\ref{Young1}) for Lorentz spaces on $\mathbb{R}$, with parameters $1+\frac{\alpha-1}{\alpha}=(\frac{d}{\alpha{\bf q}'}-\frac{\mathfrak{s}}{\alpha})+\frac{1}{\bf p}$ (recall that by hypothesis we have $\frac{\alpha}{{\bf p}}+\frac{d}{{\bf q}'}-\mathfrak{s}=2\alpha-1$ from which we easily deduce the identity $\frac{1}{{\bf p}}+(\frac{d}{\alpha{\bf q}'}-\frac{\mathfrak{s}}{\alpha})=1+\frac{\alpha-1}{\alpha}$) we thus have 
$$\left\|\left\| \int_0^t\mathfrak{p}_{t-s}\ast\mathbb{P}(\vec{f})ds\right\|_{L^{1}_\xi}\right\|_{L^{\frac{\alpha}{\alpha-1},\infty}_t} \leq C\left\| s^{-\frac{1}{\alpha}(\frac{d}{{\bf q}'}-\mathfrak{s})}\mathds{1}_{\{s>0\}}\right\|_{L^{\frac{\alpha {\bf q}'}{d-\mathfrak{s}{\bf q}'}, \infty}_t}\left\| \||\xi|^{\mathfrak{s}}|\widehat{\vf}|(\cdot, \cdot)\|_{L^{{\bf q},\infty}} \right\|_{L^{{\bf p},\infty}_t},$$
which, we can rewrite as follows (recall the definition of the norm $L^{{\bf p},\infty}_t(\mathcal{F}^{\mathfrak{s}}_{L^{{\bf q},\infty}_\xi})$ given in the expression (\ref{Def_NormFHL}) above): 
\begin{eqnarray*}
\left\| \int_0^t\mathfrak{p}_{t-s}\ast\mathbb{P}(\vec{f})ds\right\|_{L^{\frac{\alpha}{\alpha-1},\infty}_t(L^{1}_\xi)} &\leq& C\left\| s^{-(\frac{d-\mathfrak{s}{\bf q}'}{\alpha {\bf q}'})}\mathds{1}_{\{s>0\}}\right\|_{L^{\frac{\alpha {\bf q}'}{d-\mathfrak{s}{\bf q}'}, \infty}_t}\|\vf\|_{L^{{\bf p},\infty}_t(\mathcal{F}^{\mathfrak{s}}_{L^{{\bf q},\infty}_\xi})}\\
&\leq& C\|\vf\|_{L^{{\bf p},\infty}_t(\mathcal{F}^{\mathfrak{s}}_{L^{{\bf q},\infty}_\xi})},
\end{eqnarray*}
since the quantity $\| s^{-(\frac{d-\mathfrak{s}{\bf q}'}{\alpha {\bf q}'})}\mathds{1}_{\{s>0\}}\|_{L^{\frac{\alpha {\bf q}'}{d-\mathfrak{s}{\bf q}'}, \infty}_t}$ is bounded. 
%%%%%%%%%%%%%%%%%%%%%%%%%%%%%%%%%%%%%%%%%%%%%%%%%%%
\item[(3)] For the linear term we recall that we have by hypothesis the relationship $\frac{\alpha}{\mathfrak{p}}+\frac{d}{\mathfrak{q}'}=2\alpha-1$ and the same computations performed above with $\mathfrak{s}=0$ apply, we thus obtain the control: 
$$\left\|\int_0^t\mathfrak{p}_{t-s}\ast\mathbb{P}(\theta \vec{e}_d)ds\right\|_{L^{\frac{\alpha}{\alpha-1},\infty}_t(L^{1}_\xi)} \leq  C \|\theta\|_{L^{\mathfrak{p},\infty}_t(L^{\mathfrak{q},\infty}_\xi)}.$$
%%%%%%%%%%%%%%%%%%%%%%%%%%%%%%%%%%%%%%%%%%%%%%%%%%%
\item[(4)] For the bilinear term we need to study the following estimates
\begin{equation}\label{EstimationL1L1}
\left\|\int_0^t\mathfrak{p}_{t-s}\ast\mathbb{P}\left(\mathrm{div}(\vec{u}\otimes \vec{u})\right)ds\right\|_{L^{\frac{\alpha}{\alpha-1},\infty}_t(L^{1}_\xi)}\leq C\|\vu\|_{L^{\frac{\alpha}{\alpha-1},\infty}_t(L^{1}_\xi)}\|\vu\|_{L^{\frac{\alpha}{\alpha-1},\infty}_t(L^{1}_\xi)},
\end{equation}
and
\begin{equation}\label{EstimationL1Lq}
\left\|\int_0^t \mathfrak{p}_{t-s}\ast \mathrm{div}(\theta \vec{u})ds\right\|_{L^{\mathfrak{p},\infty}_t(L^{\mathfrak{q},\infty}_\xi)}\leq C\|\vu\|_{L^{\frac{\alpha}{\alpha-1},\infty}_t(L^{1}_\xi)}\|\theta\|_{L^{\mathfrak{p},\infty}_t(L^{\mathfrak{q},\infty}_\xi)}.
\end{equation}
These two terms can be studied by considering a modification of the Proposition \ref{Prop_generalquadratic} and we give the details below for the sake of completeness.
\begin{itemize}
\item[$\bullet$] For the estimate (\ref{EstimationL1L1}) we write
$$\left|\mathcal{F}\left(\int_0^t\mathfrak{p}_{t-s}\ast\mathbb{P}\left(\mathrm{div}(\vec{u}\otimes \vec{u})\right)ds\right)(\xi)\right|\leq C\displaystyle{\int_0^t} e^{-(t-s)|\xi|^\alpha}|\xi|\times \left[|\widehat{\vu}|*|\widehat{\vu}|(s,\xi)\right] ds,$$
taking $L^{1}$ norm in the expression above we have
$$\left\|\int_0^t\mathfrak{p}_{t-s}\ast\mathbb{P}\left(\mathrm{div}(\vec{u}\otimes \vec{u})\right)ds\right\|_{L^{1}_\xi} \leq C\displaystyle{\int_0^t} \left\|e^{-(t-s)|\xi|^\alpha} |\xi| \times \left[|\widehat{\vu}|*|\widehat{\vu}|(s,\xi)\right] \right\|_{L^{1}}ds,$$
now, by the H\"older inequality in the Lebesgue framework, we obtain
$$\left\|\int_0^t\mathfrak{p}_{t-s}\ast\mathbb{P}\left(\mathrm{div}(\vec{u}\otimes \vec{u})\right)ds\right\|_{L^{1}_\xi} \leq C\displaystyle{\int_0^t} \left\|e^{-(t-s)|\xi|^\alpha} |\xi| \right\|_{L^{\infty}} \left\|\left[|\widehat{\vu}|*|\widehat{\vu}|(s,\xi)\right] \right\|_{L^{1}}ds.$$
By homogeneity of the Lorentz spaces (recall the property (\ref{HomogeneiteLorentz}) above) we have the identity $\left\|e^{-(t-s)|\xi|^\alpha} |\xi|\right\|_{L^{\infty}} = \left\|e^{-|\xi|^\alpha} |\xi|\right\|_{L^{\infty}}(t-s)^{-\frac{1}{\alpha}}$, and since we have $\left\|e^{-|\xi|^\alpha} |\xi|\right\|_{L^{\infty}}<+\infty$, we can write 
$$\left\|\int_0^t\mathfrak{p}_{t-s}\ast\mathbb{P}\left(\mathrm{div}(\vec{u}\otimes \vec{u})\right)ds\right\|_{L^{1}_\xi}\leq C\int_0^t (t-s)^{-\frac{1}{\alpha}}\left\|\left[|\widehat{\vu}|*|\widehat{\vu}|(s,\xi)\right] \right\|_{L^{1}}ds.$$
Applying again the Young inequality for usual Lebesgue spaces, we have
$$\left\|\int_0^t\mathfrak{p}_{t-s}\ast\mathbb{P}\left(\mathrm{div}(\vec{u}\otimes \vec{u})\right)ds\right\|_{L^{1}_\xi} \leq C\int_0^t (t-s)^{-\frac{1}{\alpha}}\left\|\widehat{\vu}(s,\cdot)\right\|_{L^{1}}\left\|\widehat{\vu}(s,\cdot)\right\|_{L^{1}}ds.$$
We extend the function $\widehat{\vu}(s,\cdot)$ by $0$ when $s<0$ and we obtain
\begin{eqnarray*}
\left\|\int_0^t\mathfrak{p}_{t-s}\ast\mathbb{P}\left(\mathrm{div}(\vec{u}\otimes \vec{u})\right)ds\right\|_{L^{1}_\xi}\leq C\int_{-\infty}^{+\infty} \mathds{1}_{\{t-s>0\}}(t-s)^{-\frac{1}{\alpha}}\left\|\widehat{\vu}(s,\cdot)\right\|_{L^{1}}\left\|\widehat{\vu}(s,\cdot)\right\|_{L^{1}}ds\\
\leq  C\left(\mathds{1}_{\{s>0\}} s^{-\frac{1}{\alpha}}\right)\ast \left(\left\|\vu(\cdot,\cdot)\right\|_{L^{1}_\xi}\|\vu(\cdot,\cdot)\|_{L^{1}_\xi}\right)(t).
\end{eqnarray*}
Now, taking the $L^{\frac{\alpha}{\alpha-1},\infty}$ norm in the expression above we have 
\begin{eqnarray*}
\left\|\left\|\int_0^t\mathfrak{p}_{t-s}\ast\mathbb{P}\left(\mathrm{div}(\vec{u}\otimes \vec{u})\right)ds\right\|_{L^{1}_\xi}\right\|_{L^{\frac{\alpha}{\alpha-1},\infty}_t} \leq  C\left\|\left(\mathds{1}_{\{s>0\}} s^{-\frac{1}{\alpha}}\right)\ast \left(\|\vu\|_{L^{1}_\xi}\|\vu\|_{L^{1}_\xi} \right)\right\|_{L^{\frac{\alpha}{\alpha-1},\infty}_t}.
\end{eqnarray*}
At this point, we apply the Young inequality (\ref{Young1}) for Lorentz spaces (in the time variable) with $1+\frac{\alpha-1}{\alpha}=\frac{1}{\alpha}+\frac{2(\alpha-1)}{\alpha}$ to obtain
\begin{eqnarray*}
\left\|\int_0^t\mathfrak{p}_{t-s}\ast\mathbb{P}\left(\mathrm{div}(\vec{u}\otimes \vec{u})\right)ds\right\|_{L^{\frac{\alpha}{\alpha-1},\infty}_t(L^{1}_\xi)} \leq  C\left\|\mathds{1}_{\{s>0\}} s^{-\frac{1}{\alpha} }\right\|_{L^{\alpha,\infty}_t}\times \left\| \|\vu\|_{L^{1}_\xi}\|\vu\|_{L^{1}_\xi}\right\|_{L^{\frac{\alpha}{2(\alpha-1)},\infty}_t}.
\end{eqnarray*}
We remark now that $\left\|\mathds{1}_{\{s>0\}} s^{-\frac{1}{\alpha}}\right\|_{L^{\alpha,\infty}_t}<+\infty$, and we obtain the estimate 
$$\left\|\int_0^t\mathfrak{p}_{t-s}\ast\mathbb{P}\left(\mathrm{div}(\vec{u}\otimes \vec{u})\right)ds\right\|_{L^{\frac{\alpha}{\alpha-1},\infty}_t(L^{1}_\xi)}\leq  C\left\| \|\vu\|_{L^{1}_\xi}\|\vu\|_{L^{1}_\xi}\right\|_{L^{\frac{\alpha}{2(\alpha-1)},\infty}_t},$$
and by the Hölder inequality (\ref{Holder1}) for Lorentz spaces with $\frac{2(\alpha-1)}{\alpha}=\frac{\alpha-1}{\alpha}+\frac{\alpha-1}{\alpha}$, we obtain
$$\left\|\int_0^t\mathfrak{p}_{t-s}\ast\mathbb{P}\left(\mathrm{div}(\vec{u}\otimes \vec{u})\right)ds\right\|_{L^{\frac{\alpha}{\alpha-1},\infty}_t(L^{1}_\xi)}\leq  C \left\| \left\|\vu\right\|_{L^{1}_\xi}\right\|_{L^{\frac{\alpha}{\alpha-1},\infty}_t}\left\| \left\|\vu\right\|_{L^{1}_\xi}\right\|_{L^{\frac{\alpha}{\alpha-1},\infty}_t},$$
which is the wished estimate
$$\left\|\int_0^t\mathfrak{p}_{t-s}\ast\mathbb{P}\left(\mathrm{div}(\vec{u}\otimes \vec{u})\right)ds\right\|_{L^{\frac{\alpha}{\alpha-1},\infty}_t(L^{1}_\xi)}\leq  C \|\vu\|_{L^{\frac{\alpha}{\alpha-1},\infty}_t(L^{1}_\xi)} \|\vu\|_{L^{\frac{\alpha}{\alpha-1},\infty}_t(L^{1}_\xi)} .$$
%%%%%%%%%%%%%%%%%%%%%%%%%%%%%%%%%%%%%%%%%%%%%%%%%%%
\item[$\bullet$] For the control (\ref{EstimationL1Lq}) we proceed as follows: by the properties of the Fourier transform we have
$$\left|\mathcal{F}\left(\int_0^t\mathfrak{p}_{t-s}\ast\mathrm{div}(\vec{u}\theta)ds\right)(\xi)\right|\leq C\int_0^te^{-(t-s)|\xi|^\alpha}|\xi| \times \left[|\widehat{\vu}|*|\widehat{\theta}| \right] ds.$$
We take the $L^{\mathfrak{q},\infty}$ norm in order to obtain 
$$\left\|\int_0^t\mathfrak{p}_{t-s}\ast\mathrm{div}(\vec{u}\theta)ds\right\|_{L^{\mathfrak{q},\infty}_\xi} \leq \displaystyle{\int_0^t} \left\|e^{-(t-s)|\xi|^\alpha} (t-s)^\frac{1}{\alpha}|\xi| \times (t-s)^{-\frac{1}{\alpha}}\left[|\widehat{\vu}|*|\widehat{\theta}| (s,\cdot)\right] \right\|_{L^{\mathfrak{q},\infty}}ds,$$
from which we deduce
$$\left\|\int_0^t\mathfrak{p}_{t-s}\ast\mathrm{div}(\vec{u}\theta)ds\right\|_{L^{\mathfrak{q},\infty}_\xi}\leq C\displaystyle{\int_0^t} (t-s)^{-\frac{1}{\alpha}} \left\|\left[|\widehat{\vu}|*|\widehat{\theta}| (s,\cdot)\right] \right\|_{L^{\mathfrak{q},\infty}}ds,$$
as we have $\left\|e^{-(t-s)|\cdot|^\alpha}(t-s)^\frac{1}{\alpha}|\cdot|\right\|_{L^{\infty}}=\left\|e^{-|\cdot|^\alpha}|\cdot|\right\|_{L^{\infty}}<+\infty$. Applying the Young inequality for Lorentz spaces (see (\ref{Young2})), we obtain
$$\left\|\int_0^t\mathfrak{p}_{t-s}\ast\mathrm{div}(\vec{u}\theta)ds\right\|_{L^{\mathfrak{q},\infty}_\xi}\leq C\int_0^t (t-s)^{-\frac{1}{\alpha}}\left\|\widehat{\vu}(s,\cdot)\right\|_{L^{1}}\left\|\widehat{\theta}(s,\cdot)\right\|_{L^{\mathfrak{q},\infty}}ds.$$
We extend the functions $\widehat{\vu}(s,\cdot)$ and $\widehat{\theta}(s,\cdot)$ by $0$ when $s<0$, we can write 
\begin{eqnarray*}
\left\|\int_0^t\mathfrak{p}_{t-s}\ast\mathrm{div}(\vec{u}\theta)ds\right\|_{L^{\mathfrak{q},\infty}_\xi}\leq C\int_{-\infty}^{+\infty}  \mathds{1}_{\{t-s>0\}}(t-s)^{-\frac{1}{\alpha}}\left\|\widehat{\vu}(s,\cdot)\right\|_{L^{1}}\left\|\widehat{\theta}(s,\cdot)\right\|_{L^{\mathfrak{q},\infty}}ds\\
\leq  C\left(\mathds{1}_{\{s>0\}} s^{-\frac{1}{\alpha}}\right)\ast \left(\left\|\widehat{\vu}(\cdot,\cdot)\right\|_{L^{1}}\left\|\widehat{\theta}(\cdot,\cdot)\right\|_{L^{\mathfrak{q},\infty}}\right)(t).
\end{eqnarray*}
Now, taking the $L^{\mathfrak{p},\infty}$ norm in the expression above we have 
\begin{eqnarray*}
\left\|\left\|\int_0^t\mathfrak{p}_{t-s}\ast\mathrm{div}(\vec{u}\theta)ds\right\|_{L^{\mathfrak{q},\infty}_\xi}\right\|_{L^{\mathfrak{p},\infty}_t} \leq  C\left\|\left(\mathds{1}_{\{s>0\}} s^{-\frac{1}{\alpha}}\right)\ast \left(\left\|\widehat{\vu}\right\|_{L^{1}}\left\|\widehat{\theta}\right\|_{L^{\mathfrak{q},\infty}}\right)\right\|_{L^{\mathfrak{p},\infty}_t}.
\end{eqnarray*}
At this point, we apply the Young inequality (\ref{Young1}) for Lorentz spaces (in the time variable) with $1+\frac{1}{\mathfrak{p}}=\frac{1}{\mathfrak{r}}+\frac{1}{\rho}$ and $1<\mathfrak{r}, \rho<+\infty$: 
\begin{eqnarray*}
\left\|\int_0^t\mathfrak{p}_{t-s}\ast\mathrm{div}(\vec{u}\theta)ds\right\|_{L^{\mathfrak{p},\infty}_t(L^{\mathfrak{q}, \infty}_\xi)} &\leq & C\left\|\mathds{1}_{\{s>0\}} s^{-\frac{1}{\alpha}}\right\|_{L^{\mathfrak{r},\infty}_t}\\
&&\times \left\| \left\|\widehat{\vu}(\cdot,\cdot)\right\|_{L^{1}}\left\|\widehat{\theta}(\cdot,\cdot)\right\|_{L^{\mathfrak{q},\infty}}\right\|_{L^{\rho,\infty}_t}.
\end{eqnarray*}
Note that if we set $\frac{1}{\rho}=\frac{\alpha-1}{\alpha}+\frac{1}{\mathfrak{p}}$, from the relationship  $1+\frac{1}{\mathfrak{p}}=\frac{1}{\mathfrak{r}}+\frac{1}{\rho}$, we have $\frac{1}{\alpha}=\frac{1}{\mathfrak{r}}$ (recall that $1<\alpha<2$) and we deduce that
$\left\|\mathds{1}_{\{s>0\}} s^{-\frac{1}{\alpha}}\right\|_{L^{\mathfrak{r},\infty}_t}=\left\|\mathds{1}_{\{s>0\}} s^{-\frac{1}{\alpha}}\right\|_{L^{\alpha,\infty}_t}<+\infty$, and we obtain the estimate 
$$\left\|\int_0^t\mathfrak{p}_{t-s}\ast\mathrm{div}(\vec{u}\theta)ds\right\|_{L^{\mathfrak{p},\infty}_t(L^{\mathfrak{q}, \infty}_\xi)} \leq  C \left\| \left\|\widehat{\vu}(\cdot,\cdot)\right\|_{L^{1}}\left\|\widehat{\theta}(\cdot,\cdot)\right\|_{L^{\mathfrak{q},\infty}}\right\|_{L^{\rho,\infty}_t}.$$
Finally, we use the Hölder inequality (\ref{Holder1}) for Lorentz spaces with $\frac{1}{\rho}=\frac{\alpha-1}{\alpha}+\frac{1}{\mathfrak{p}}$ to get
$$\left\|\int_0^t\mathfrak{p}_{t-s}\ast\mathrm{div}(\vec{u}\theta)ds\right\|_{L^{\mathfrak{p},\infty}_t(L^{\mathfrak{q}, \infty}_\xi)}  \leq  C \left\| \left\|\widehat{\vu}(\cdot,\cdot)\right\|_{L^{1}}\right\|_{L^{\frac{\alpha}{\alpha-1},\infty}_t} \left\|\left\|\widehat{\theta}(\cdot,\cdot)\right\|_{L^{\mathfrak{q},\infty}}\right\|_{L^{\mathfrak{p},\infty}_t},$$
which we rewrite as follows: 
$$\left\|\int_0^t\mathfrak{p}_{t-s}\ast\mathrm{div}(\vec{u}\theta)ds\right\|_{L^{\mathfrak{p},\infty}_t(L^{\mathfrak{q}, \infty}_\xi)}\leq  C \left\|\vu\right\|_{L^{\frac{\alpha}{\alpha-1},\infty}_t(L^{1}_\xi)}  \left\| \theta\right\|_{L^{\mathfrak{p},\infty}_t(L^{\mathfrak{q},\infty}_\xi)}.$$
%%%%%%%%%%%%%%%%%%%%%%%%%%%%%%%%%%%%%%%%%%%%%%%%%%%
\end{itemize}
\end{itemize}
With all the estimates (\ref{Estimation_DonneesInitialesAlpha}), (\ref{Estimation_ForceAlpha}), (\ref{Estimation_LineaireAlpha}) and (\ref{Estimation_BilineaireAlpha}) at our disposal, if the quantity $\|\vu_0\|_{\dot{\mathscr{B}}^{-(\alpha-1), \infty}_{L^{1}_\xi}}+\|\theta_0\|_{ \dot{\mathscr{B}}^{-\frac{\alpha}{\mathfrak{p}}, \infty}_{L^{\mathfrak{q},\infty}_\xi}}+\|\vec{f}\|_{L^{{\bf p},\infty}_t(\mathcal{F}^{s}_{L^{{\bf q},\infty}_\xi})}$ is small enough, then we can again apply the Lemma \ref{Lemme_Banach_Picard} to obtain a global mild solution $(\vec{u},\theta)$ to the fractional dissipative Navier-Stokes-Boussinesq system (\ref{eq_NSBalpha}) such that we have $(\vec{u},\theta)\in L^{\frac{\alpha}{\alpha-1},\infty}_t(L^1_\xi)\times L^{\mathfrak{p},\infty}_t(L^{\mathfrak{q},\infty}_\xi)$. The Theorem \ref{thm_boussinesq_L1_alpha} is now completely proven. $\hfill \blacksquare$
%%%%%%%%%%%%%%%%%%%%%%%%%%%%%%%%%%%%%%%%%%%%%%%%%%%
\appendix
%%%%%%%%%%%%%%%%%%%%%%%%%%%%%%%%%%%%%%%%%%%%%%%%%%%
\section{Proof of the Proposition \ref{prop_max_fourier_besov}}\label{Appendix1}
We aim to prove the following control $\|\vec{\psi}\|_{\dot{\mathscr{B}}^{-\beta, \infty}_{L^{1}_\xi}}\leq C \|\vec{\psi}\|_{E}$ for all function $\vec{\psi}\in E$. Recall that by the formula (\ref{NormFBL}) we have $\|\vec{\psi}\|_{\dot{\mathscr{B}}^{-\beta, \infty}_{L^{1}_\xi}}=\underset{j\in \mathbb{Z}}{\sup}\;2^{-j\beta}\left\|\mathds{1}_{\{2^j\leq |\xi| \leq 2^{j+1}\}}\widehat{\;\vec{\psi}\;}(\cdot)\right\|_{L^{1}}$. Consider now $\vec{\psi}\in E$ and fix $j\in \mathbb{Z}$. By a change of variables, we have the identity
$$\|\mathds{1}_{\{2^j\leq |\xi|\leq 2^{j+1}\}}\widehat{\vec{\psi}}(\cdot)\|_{L^1}=2^{jd} \|\mathds{1}_{\{1\leq |\xi|\leq 2\}}\widehat{\vec{\psi}}(2^j\xi)\|_{L^1},$$
which we rewrite as 
\begin{eqnarray*}
\|\mathds{1}_{\{2^j\leq |\xi|\leq 2^{j+1}\}}\widehat{\vec{\psi}}(\cdot)\|_{L^1}&=&2^{j\beta} \|\mathds{1}_{\{1\leq |\xi| \leq2\}}2^{j(d-\beta)}\widehat{\vec{\psi}}(2^j\xi)\|_{L^1}\\
&=&2^{j\beta} \|\mathds{1}_{\{1\leq |\xi|\leq 2\}}\widehat{\vec{\psi}}_j(\cdot)\|_{L^1},
\end{eqnarray*}
with $\widehat{\vec{\psi}}_j(\xi)=2^{j(d-\beta)}\widehat{\vec{\psi}}(2^j\xi)$. Now, by the first hypothesis we obtain the control 
$$\|\mathds{1}_{\{2^j\leq |\xi|\leq 2^{j+1}\}}\widehat{\vec{\psi}}(\cdot)\|_{L^1}=2^{j\beta} \|\mathds{1}_{\{1\leq |\xi|\leq 2\}}\widehat{\vec{\psi}}_j(\cdot)\|_{L^1}\leq C 2^{j\beta}\|\widecheck{\widehat{\vec{\psi}_j}} \,(\cdot)\|_E,$$
noting that $\widecheck{\widehat{\vec{\psi}_j}}(x)=2^{-j\beta}\psi(2^{-j}x)$, we obtain
$$\|\mathds{1}_{\{2^j\leq |\xi|\leq 2^{j+1}\}}\widehat{\vec{\psi}}(\cdot)\|_{L^1}\leq C 2^{j\beta}\|2^{-j\beta}\vec{\psi}(2^{-j}\cdot)\|_E,$$
and at this point we use the homogeneity hypothesis to write
$$\|\mathds{1}_{\{2^j\leq |\xi|\leq 2^{j+1}\}}\widehat{\vec{\psi}}(\cdot)\|_{L^1}\leq C 2^{j\beta}\|2^{-j\beta}\vec{\psi}(2^{-j}\cdot)\|_E= C 2^{j\beta}\|\vec{\psi}\|_E,$$
from which we obtain the uniform estimate
$$2^{-j\beta}\|\mathds{1}_{\{2^j\leq |\xi|\leq 2^{j+1}\}}\widehat{\vec{\psi}}(\cdot)\|_{L^1}\leq  C\|\vec{\psi}\|_E,$$
and we finally obtain 
$$\|\vec{\psi}\|_{\dot{\mathscr{B}}^{-\beta,\infty}_{L^{1}_\xi}}=\underset{j\in \mathbb{Z}}{\sup}\;2^{-j\beta}\left\|\mathds{1}_{\{2^j\leq |\xi| \leq 2^{j+1}\}}\widehat{\;\vec{\psi}\;}(\cdot)\right\|_{L^{1}}\leq C\|\vec{\psi}\|_E.$$
We will prove now the space inclusion $\dot{\mathscr{B}}^{-\frac{\alpha}{p}, \infty}_{L^{q,\infty}_\xi}(\mathbb{R}^d)\subset \dot{\mathscr{B}}^{-\beta, \infty}_{L^{1}_\xi}(\mathbb{R}^d)$ under the condition $\frac{\alpha}{p}+\frac{d}{q'}=\beta$. For this we first remark that, definition of the space $\dot{\mathscr{B}}^{-\frac{\alpha}{p}, \infty}_{L^{q,\infty}_\xi}(\mathbb{R}^d)$ given in the expression (\ref{NormFBL}), we have
$$\|\vf\|_{\dot{\mathscr{B}}^{-\frac{\alpha}{p}, \infty}_{L^{q,\infty}_\xi}}=\underset{j\in \mathbb{Z}}{\sup}\; 2^{-j\frac{\alpha}{p}}\left\|\mathds{1}_{\{2^j\leq |\xi| \leq 2^{j+1}\}}\widehat{\;\vf \;}(\cdot)\right\|_{L^{q,\infty}},$$
thus, in the particular case when $j=0$, we have 
$$\|\vf\|_{\dot{\mathscr{B}}^{-\frac{\alpha}{p}, \infty}_{L^{q,\infty}_\xi}}=\underset{j\in \mathbb{Z}}{\sup}\; 2^{-j\frac{\alpha}{p}}\left\|\mathds{1}_{\{2^j\leq |\xi| \leq 2^{j+1}\}}\widehat{\;\vf \;}(\cdot)\right\|_{L^{q,\infty}}\geq \left\|\mathds{1}_{\{1\leq |\xi| \leq 2\}}\widehat{\;\vf \;}(\cdot)\right\|_{L^{q,\infty}},$$
noting that $\left\|\mathds{1}_{\{1\leq |\xi| \leq 2\}}\widehat{\;\vf \;}(\cdot)\right\|_{L^{q,\infty}}\geq \left\|\mathds{1}_{\{1\leq |\xi| \leq 2\}}\widehat{\;\vf \;}(\cdot)\right\|_{L^{1}}$ (recall \cite[Proposici\'on 1.2.14]{Chamorro_Lorentz}), we obtain 
$$\|\vf\|_{\dot{\mathscr{B}}^{-\frac{\alpha}{p}, \infty}_{L^{q,\infty}_\xi}}\geq \left\|\mathds{1}_{\{1\leq |\xi| \leq 2\}}\widehat{\;\vf \;}(\cdot)\right\|_{L^{1}},$$
which is the first hypothesis of the Proposition \ref{prop_max_fourier_besov}. Now by homogeneity we have 
$$\|\lambda^\beta f(\lambda \cdot)\|_{\dot{\mathscr{B}}^{-\frac{\alpha}{p}, \infty}_{L^{q,\infty}_\xi}}=\lambda^\beta\| f(\lambda \cdot)\|_{\dot{\mathscr{B}}^{-\frac{\alpha}{p}, \infty}_{L^{q,\infty}_\xi}}=\lambda^{\frac{\alpha}{p}+\frac{d}{q'}}\| f(\lambda \cdot)\|_{\dot{\mathscr{B}}^{-\frac{\alpha}{p}, \infty}_{L^{q,\infty}_\xi}}\simeq \|f\|_{\dot{\mathscr{B}}^{-\frac{\alpha}{p}, \infty}_{L^{q,\infty}_\xi}},$$
which is the second hypothesis of the Proposition \ref{prop_max_fourier_besov} and thus we obtain the wished space inclusion $\dot{\mathscr{B}}^{-\frac{\alpha}{p}, \infty}_{L^{q,\infty}_\xi}(\mathbb{R}^d)\subset \dot{\mathscr{B}}^{-\beta, \infty}_{L^{1}_\xi}(\mathbb{R}^d)$ when  $\frac{\alpha}{p}+\frac{d}{q'}=\beta$, and this ends the proof of this proposition. $\hfill \blacksquare$\\
%%%%%%%%%%%%%%%%%%%%%%%%%%%%%%%%%%%%%%%%%%%%%%%%%%%
\section{Fourier-Herz spaces and parabolic Morrey spaces}\label{Appendix2}
In the Theorem \ref{thm_boussinesq_L1_alpha} above, we used the space $L^{\frac{\alpha}{\alpha-1}, \infty}_t(L^1_\xi)$ as a resolution space for the velocity field $\vu$. In this section we will compare this particular resolution space (based in a Fourier approach) with the parabolic Morrey spaces that were used in \cite{Chamorro_Mansais1} (based in a real variable approach) and we will see that the Fourier-Herz framework considered here can not be compared to the parabolic Morrey setting when studying critical mild solutions for the fractional dissipative Navier-Stokes-Boussinesq system  (\ref{eq_NSBalpha}). To this end, we first recall the definition of the parabolic Morrey spaces: 
\begin{Definition}
For $1\leq \mathfrak{p}\leq \mathfrak{q}<+\infty$ and for $1<\alpha<2$, the parabolic Morrey space $\mathcal{M}_\alpha^{\mathfrak{p},\mathfrak{q}}([0,+\infty[\times \mathbb{R}^d)$ is defined as the space of locally integrable functions $\vec{\psi}:[0,+\infty[\times \mathbb{R}^d\longrightarrow \mathbb{R}^d$ such that the following norm is finite
$$\|\vec{\psi}\|_{\mathcal{M}_\alpha^{\mathfrak{p},\mathfrak{q}}}=\underset{r>0}{\mathrm{sup}} \ \underset{(t,x)\in [0,+\infty[ \times \mathbb{R}^d}{\mathrm{sup}} \ \frac{1}{r^{(d+\alpha)(\frac{1}{\mathfrak{p}}-\frac{1}{\mathfrak{q}})}} \left(\displaystyle{\iint_{\{|t-s|^\frac{1}{\alpha}+|x-y|<r\}}}|\vec{\psi}(s,y)|^\mathfrak{p}dyds \right)^\frac{1}{\mathfrak{p}} <+\infty.$$
\end{Definition}
Morrey spaces are \emph{big} functional spaces in the sense that, under some condition over the indexes, they can contain Lebesgue spaces and Lorentz spaces. In our particular case, we have the following space inclusion: 
\begin{equation}\label{LorentzHerzParaMorrey}
L^{\frac{\alpha}{\alpha-1},\infty}_t([0,+\infty[, L^1_\xi(\mathbb{R}^d))\subset \mathcal{M}_\alpha^{\mathfrak{p},\frac{d+\alpha}{\alpha-1}}([0,+\infty[\times \mathbb{R}^d), \quad \mbox{if}\quad  1\leq \mathfrak{p}<\tfrac{\alpha}{\alpha-1}.
\end{equation}
Indeed, consider a vector field $\vec{\psi}$ such that $\vec{\psi}\in L^{\frac{\alpha}{\alpha-1}, \infty }_t(L^1_\xi)$, we then write
\begin{eqnarray*}
\|\vec{\psi}\|_{\mathcal{M}_\alpha^{\mathfrak{p},\frac{d+\alpha}{\alpha-1}}}&=&\underset{r>0}{\mathrm{sup}} \ \underset{(t,x)\in [0,+\infty[ \times \mathbb{R}^d}{\mathrm{sup}} \ \frac{1}{r^{(d+\alpha)(\frac{1}{\mathfrak{p}}-\frac{\alpha-1}{d+\alpha})}} \left(\displaystyle{\iint_{\{|t-s|^\frac{1}{\alpha}+|x-y|<r\}}}|\vec{\psi}(s,y)|^\mathfrak{p}dyds \right)^\frac{1}{\mathfrak{p}}\\
&\leq &\underset{r>0}{\mathrm{sup}} \ \underset{(t,x)\in [0,+\infty[ \times \mathbb{R}^d}{\mathrm{sup}} \ \frac{1}{r^{(d+\alpha)(\frac{1}{\mathfrak{p}}-\frac{\alpha-1}{d+\alpha})}} \left(\displaystyle{\int_{\{|t-s|<r^\alpha\}}}\|\vec{\psi}(s,\cdot)\|_{L^\infty}^\mathfrak{p} r^d ds \right)^\frac{1}{\mathfrak{p}},
\end{eqnarray*}
and recalling that $\|\vec{\psi}(s,\cdot)\|_{L^\infty}\leq \|\widehat{\vec{\psi}}(s,\cdot)\|_{L^1}=\|\vec{\psi}(s,\cdot)\|_{L^1_\xi}$, we obtain 
$$\|\vec{\psi}\|_{\mathcal{M}_\alpha^{\mathfrak{p},\frac{d+\alpha}{\alpha-1}}}\leq \underset{r>0}{\mathrm{sup}} \ \underset{(t,x)\in [0,+\infty[ \times \mathbb{R}^d}{\mathrm{sup}} \ \frac{r^{\frac{d}{\mathfrak{p}}}}{r^{(d+\alpha)(\frac{1}{\mathfrak{p}}- \frac{\alpha-1}{d+\alpha})}} \left(\int_{\{|t-s|<r^\alpha\}}\|\vec{\psi}(s,\cdot)\|_{L^1_\xi}^\mathfrak{p}  ds \right)^\frac{1}{\mathfrak{p}}.$$
Since we have the control (see \cite[Proposici\'on 1.1.10]{Chamorro_Lorentz})
$$\int_{\{|t-s|<r^{\alpha}\}}\|\vec{\psi}(s,\cdot)\|_{L^1_\xi}^{\mathfrak{p}}ds\leq C (r^\alpha)^{1-\frac{\mathfrak{p}(\alpha-1)}{\alpha}}\|\|\vec{\psi}(s,\cdot)\|_{L^1_\xi}\|_{L^{\frac{\alpha}{\alpha-1},\infty}_t}^\mathfrak{p},$$
as long as $1\leq \mathfrak{p}<\frac{\alpha}{\alpha-1}$, then we have 
 $$\|\vec{\psi}\|_{\mathcal{M}_\alpha^{\mathfrak{p},\frac{d+\alpha}{\alpha-1}}}\leq C \underset{r>0}{\mathrm{sup}} \ \underset{(t,x)\in [0,+\infty[ \times \mathbb{R}^d}{\mathrm{sup}} \ \frac{r^{\frac{d}{\mathfrak{p}}}r^{\frac{\alpha}{\mathfrak{p}}-(\alpha-1)}}{r^{(d+\alpha)(\frac{1}{\mathfrak{p}}-\frac{\alpha-1}{d+\alpha})}}\|\vec{\psi}\|_{L^{\frac{\alpha}{\alpha-1}, \infty}_t(L^1_\xi)},$$
 and we obtain 
$$\|\vec{\psi}\|_{\mathcal{M}_\alpha^{\mathfrak{p},\frac{d+\alpha}{\alpha-1}}}\leq C\|\vec{\psi}\|_{L^{\frac{\alpha}{\alpha-1},\infty}_t(L^1_\xi)},$$
from which we deduce the announced space embedding (\ref{LorentzHerzParaMorrey}) as long as we have the condition $1\leq \mathfrak{p}<\frac{\alpha}{\alpha-1}$.\\
 
With these computations at hand, we are quite tempted to deduce that the conclusion of the Theorem \ref{thm_boussinesq_L1_alpha} stated above is just a particular case of the results obtained in \cite{Chamorro_Mansais1} where we used the parabolic Morrey spaces $\mathcal{M}_\alpha^{\mathfrak{p},\frac{d+\alpha}{\alpha-1}}$. Note however that the embedding (\ref{LorentzHerzParaMorrey}) is valid as long as the parameter $\mathfrak{p}$ is \emph{small} (\emph{i.e.} when $1\leq \mathfrak{p}<\frac{\alpha}{\alpha-1}$) and in order to close the fixed point argument in the critical space $\mathcal{M}_\alpha^{\mathfrak{p},\frac{d+\alpha}{\alpha-1}}$ for the system (\ref{eq_NSBalpha}) we established in \cite{Chamorro_Mansais1} the condition $\frac{3\alpha-2}{\alpha-1}<\mathfrak{p}\leq \frac{d+\alpha}{\alpha-1}$: due to the nonlinear terms of the equation (\ref{eq_NSBalpha}) we know how to construct mild solutions in the space $\mathcal{M}_\alpha^{\mathfrak{p},\frac{d+\alpha}{\alpha-1}}$ if the parameter $\mathfrak{p}$ is \emph{not small} (note that $\frac{\alpha}{\alpha-1}<\frac{3\alpha-2}{\alpha-1}$ since $1<\alpha<2$). \\

This shows that, even though we have a \emph{generic} embedding between the resolution space $L^{\frac{\alpha}{\alpha-1}, \infty}_t(L^1_\xi)$ and the resolution space $\mathcal{M}_\alpha^{\mathfrak{p},\frac{d+\alpha}{\alpha-1}}$, the interval of validity of the inclusion $L^{\frac{\alpha}{\alpha-1}, \infty}_t(L^1_\xi)\subset \mathcal{M}_\alpha^{\mathfrak{p},\frac{d+\alpha}{\alpha-1}}$ does not allow to compare -to the best of our knowledge- these two functional settings when studying mild solutions for the fractional dissipative Navier-Stokes-Boussinesq system (\ref{eq_NSBalpha}). In other words, the Fourier based approach developed here can not be compared to the real variable approach studied in our previous work \cite{Chamorro_Mansais1}. \\

%%%%%%%%%%%%%%%%%%%%%%%%%%%%%%%%%%%%%%%%%%%%%%%%%%%

%%%%%%%%%%%%%%%%%%%%%%%%%%%%%%%%%%%%%%%%%%%%%%%%%%%
\noindent {\bf Statement.} All authors have contributed to the manuscript substantially and have agreed to the final submitted version. This work does not have any conflict of interest.
%%%%%%%%%%%%%%%%%%%%%%%%%%%%%%%%%%%%%%%%%%%%%%%%%%%

\end{document}